\documentclass[10pt,english,a4paper]{amsart}
\pdfoutput=1

\usepackage[latin1]{inputenc}
\usepackage[T1]{fontenc}
\usepackage{lmodern}
\usepackage{babel}
\usepackage{amsmath,amssymb,amsthm}
\usepackage[marginratio={1:1,1:1},totalwidth=400pt,totalheight=600pt]{geometry}
\usepackage{microtype}
\usepackage{hyperref}
\usepackage{url}
\usepackage{color}
\usepackage[all]{xy}

\usepackage{tikz}
\usepackage{tikz-cd}
\usetikzlibrary{matrix,calc,arrows}

\newtheorem{theorem}{Theorem}[section]
\newtheorem{lemma}[theorem]{Lemma}
\newtheorem{proposition}[theorem]{Proposition}
\newtheorem{corollary}[theorem]{Corollary}

\theoremstyle{definition}
\newtheorem{definition}[theorem]{Definition}
\newtheorem{example}[theorem]{Example}
\newtheorem{question}[theorem]{Question}
\newtheorem{remark}[theorem]{Remark}

\newcommand{\Z}{\mathbb{Z}}
\newcommand{\N}{\mathbb{N}}

\newcommand{\T}{\mathbb{T}}
\newcommand{\F}{\mathbb{F}}
\newcommand{\C}{\mathbb{C}}
\newcommand{\R}{\mathbb{R}}
\renewcommand{\L}{\mathcal{L}}

\title{$C^*$-irreducibility for reduced twisted group $C^*$-algebras}

\author{Erik B\'edos}
\address{Department of Mathematics\\University of Oslo\\NO-0316 Oslo\\Norway}
\email{bedos@math.uio.no}

\author{Tron Omland}
\address{Norwegian National Security Authority}
\email{tron.omland@gmail.com}

\date{\today}

\begin{document}

\maketitle

\begin{abstract}
We study $C^*$-irreducibility of inclusions of reduced twisted group $C^*$-algebras and of reduced group $C^*$-algebras.
We characterize  $C^*$-irreducibility in the case of an inclusion arising from a normal subgroup, and 
exhibit many new examples of $C^*$-irreducible inclusions.
\end{abstract}

\section{introduction}

Let $A$ be a unital $C^*$-algebra and 
 $B\subseteq A$ be a unital inclusion, i.e., $B$ is a $C^*$-subalgebra of $A$ containing the unit $1$ of $A$.
 A $C^*$-algebra $C$ such that
$B\subseteq C\subseteq A$ 
 is called an intermediate $C^*$-algebra of $B\subseteq A$.
Inspired by several previous works, including \cite{ILP}, \cite{AK} and \cite{AU}, 
R\o rdam  recently introduced \cite{Rordam} the notion of $C^*$-irreducibility for such inclusions of $C^*$-algebras: 
$B\subseteq A$ is said to be \emph{$C^*$-irreducible} 
if every intermediate $C^*$-algebra of $B\subseteq A$ is simple.
Besides giving an intrinsic characterization of $C^*$-irreducibility, R\o rdam presents in \cite{Rordam} several examples of $C^*$-irreducible inclusions arising in various settings, such as groups, dynamical systems, inductive limits and tensor products. New examples have since appeared in \cite{ER, KKLRU, KS}.
The related problem of determining all intermediate $C^*$-algebras of a given inclusion has attracted a lot of attention over the years, also in its von Neumann algebraic version. As a sample, we refer to \cite{Amrutam, CS, Cameron-Smith, Choda, ILP, Suz, Zac, Zsi}.

Some other properties of unital inclusions of $C^*$-algebras have also been studied. In \cite{Ursu}, Ursu 
says that $B \subseteq A$ is \emph{relatively simple}  if any unital completely positive map on $A$ which is a $*$-homomorphism on $B$ is faithful on $A$.
Every relatively simple inclusion is $C^*$-irreducible \cite[Proposition~3.6]{Ursu}. On the other hand, inspired by the terminology used for von~Neumann algebras, a unital inclusion $B\subseteq A$ of $C^*$-algebras is said to be \emph{irreducible} if $B' \cap A=\C1$.  Equivalently, as we will see later, this amounts to require that every intermediate $C^*$-algebra of $B\subseteq A$ has trivial center.
Every $C^*$-irreducible inclusion is irreducible, as explained in \cite[Remark~3.8]{Rordam}, but the converse does not hold in general.
Another property of a unital inclusion $B\subseteq A$
  that has been of interest in the past is the \emph{relative Dixmier property},
 which says that for any $a \in A$ the norm closure of the convex hull of $\{ uau^*: u \text{ unitary in } B\}$ contains a scalar (see \cite{Popa, Rordam} and references therein). It is stronger than $C^*$-irreducibility when $A$ has a faithful tracial state, and always stronger than irreducibility, cf.~\cite[Proposition~3.12]{Rordam}.

Now, let $G$ be a discrete group and $\sigma$ a (circle-valued) two-cocycle on $G$. One may then form the associated reduced twisted group $C^*$-algebra $C_r^*(G, \sigma)$ and the twisted group von Neumann algebra $\L(G, \sigma)$.  Our main focus in this article is to study $C^*$-irreducibility  of inclusions of the form $C^*_r(H,\sigma)\subseteq C^*_r(G,\sigma)$, where $H$ is a subgroup of $G$. A necessary condition for this to happen is clearly that both $C^*_r(H,\sigma)$ and $C^*_r(G,\sigma)$ are simple. 
When $\sigma$ is trivial, this means that $H$ and $G$ have to be $C^*$-simple.
However, despite the recent 
breakthroughs  \cite{BKKO, Haa, Ken} in the theory of $C^*$-simple groups,
the problem of determining when a reduced twisted group $C^*$-algebra is simple is more complicated,
see for example our discussion in \cite{SUT}. This indicates that obtaining an intrinsic characterization of the $C^*$-irreducibility of $C^*_r(H,\sigma)\subseteq C^*_r(G,\sigma)$ is probably also a challenging problem in general. 
Anyhow, when $H$ is normal in $G$, we are able to show in Theorem \ref{relative-simplicity} that $C^*_r(H,\sigma)\subseteq C^*_r(G,\sigma)$ is $C^*$-irreducible if and only if $C_r^*(H, \sigma)$ is simple and $(H\leq G, \sigma)$ satisfies the so-called \emph{relative Kleppner condition}. Moreover, if $H$ is FC-hypercentral or $C^*$-simple, then  we get that $C^*_r(H,\sigma)\subseteq C^*_r(G,\sigma)$ is $C^*$-irreducible if and only if  $(H\leq G, \sigma)$ satisfies the relative Kleppner condition, if and only if $C^*_r(H,\sigma)\subseteq C^*_r(G,\sigma)$ satisfies the relative Dixmier property. The relative Kleppner condition, which is of purely combinatorial nature, has its origin from Kleppner's work \cite{Kleppner} on factoriality of twisted group von Neumann algebras. As we show in Proposition \ref{kleppner-commutant}, 
it is equivalent to the irreducibility of  $C_r^*(H, \sigma) \subseteq C_r^*(G, \sigma)$ (resp.~$\L(H, \sigma) \subseteq \L(G, \sigma)$).
When $\sigma$ is trivial, this condition just says that $G$ is icc relatively to $H$, and we obtain in Theorem \ref{FCGH} that if $H$ is normal in $G$, then $C^*_r(H)\subseteq C^*_r(G)$ is $C^*$-irreducible if and only if $H$ is $C^*$-simple and the centralizer $C_G(H)$ of $H$ in $G$ is trivial. As shown by Ursu in \cite{Ursu}, this is in turn equivalent to $C_r^*(H) \subseteq C_r^*(G)$ being relatively simple.  Our Theorem \ref{FCGH} should be seen in light of \cite[Theorem 1.4]{BKKO}, which says that $G$ is $C^*$-simple if and only if both $H$ and $C_G(H)$ are $C^*$-simple. 

We pay some special attention to the case where $H$ is a normal subgroup which is \emph{prime}, in the sense that the FC-center of $H$ is torsion-free, cf.~Definition \ref{primedef}. Under this assumption, we characterize  when $(H \leq G, \sigma)$ satisfies Kleppner's condition in Theorem \ref{sigma-centralizer2}, and show in Corollary \ref{prime-irred} that $C_r^*(H, \sigma) \subseteq C_r^*(G, \sigma)$ is $C^*$-irreducible if and only if $C_r^*(H, \sigma)$ is simple and 
the \emph{twisted} centralizer $C^\sigma_G(H)$ of $H$ in $G$ is trivial. This last result takes an even simpler form if $H$ is also assumed to be FC-hypercentral, or if $H$ is $C^*$-simple, cf.~Corollary \ref{FC-C*-simple-sigma}.

When $H$ is a normal subgroup of $G$, our approach is to decompose $C_r^*(G, \sigma)$ as the reduced twisted $C^*$-crossed product 
of $C_r^*(H, \sigma)$ by a twisted action of the quotient group $G/H$, and combine this fact with a study of $C^*$-irreducibility for reduced twisted $C^*$-crossed products.  As a bonus,  the intermediate $C^*$-algebras of  the inclusion $C_r^*(H, \sigma) \subseteq C_r^*(G, \sigma)$ can be described by making use of Cameron and Smith's result \cite[Theorem~4.4]{Cameron-Smith} for simple reduced twisted $C^*$-crossed products. When $H$ is not normal in $G$, such a decomposition is not available, and we point out in~Remark \ref{counterexample} that Theorem \ref{relative-simplicity} does not necessarily hold in this case, even if $\sigma$ is trivial.

Our paper is organized as follows. Section \ref{prel} is devoted to some preliminary material on reduced twisted group $C^*$-algebras and twisted group von Neumann algebras. In the next section we discuss when a group $G$ is  icc relatively to a subgroup $H$ and introduce the relative Kleppner condition for $(H\leq G, \sigma)$, where $\sigma$ is a two-cocycle on $G$. 
 The main goal of Section \ref{irred-prim} is to show the equivalence of the relative Kleppner condition being satisfied and the irreducibility of the associated inclusions of twisted group algebras, as mentioned above. We also show that this is equivalent to the primeness of every intermediate $C^*$-algebra of  $C_r^*(H, \sigma) \subseteq C_r^*(G, \sigma)$. In Section \ref{irred-crpro} we first point out that R\o rdam's result \cite[Theorem 5.8]{Rordam}, which characterizes the $C^*$-irreducibility of the inclusion $A \subseteq A \rtimes_r G$, where $A \rtimes_r G$ is the reduced $C^*$-crossed product associated to an action of $G$ on a unital $C^*$-algebra $A$, is still valid in the case of a twisted action. Assuming that $H$ is a normal subgroup of $G$, we obtain in Theorem \ref{outer-quotient} a characterization of the $C^*$-irreducibility  of   $ A \rtimes_r H \subseteq  A \rtimes_r G$ for a twisted action. Section \ref{irred-rtgc} contains our characterization of the $C^*$-irreducibility of $C_r^*(H, \sigma) \subseteq C_r^*(G, \sigma)$, hence of $C_r^*(H) \subseteq C_r^*(G)$,
and some of the consequences that may be drawn from it.
We illustrate our findings in Section \ref{irred-examples}, where we present various new examples of $C^*$-irreducible inclusions, involving noncommutative tori, the discrete Heisenberg group, the braid group on infinitely many strands, and wreath products. In Section \ref{irred-trees}, we apply our results to produce  $C^*$-irreducible inclusions associated to groups acting on trees, e.g., amalgamated free products and HNN-extensions. 
Our final section is a sequel where we have included a related result  on the simplicity of $C_r^*(G, \sigma)$ in the presence of a normal subgroup.  

\section{Preliminaries}\label{prel}
Let $G$ be a discrete group with identity $e$. By a (normalized) \emph{two-cocycle on $G$} we will mean in this article a 
map $\sigma\colon G\times G\to\T$ satisfying
\[
\begin{gathered}
	\sigma(g,h)\sigma(gh,k)=\sigma(g,hk)\sigma(h,k), \\
	\sigma(g,e)=\sigma(e,g)=1
\end{gathered}
\]
for all $g,h,k\in G$. 
We assume throughout this paper that $G$ and  $\sigma$ are given. 
We recall that if $\sigma'$ is also a two-cocycle on $G$, then $\sigma'$ is said to be \emph{similar to $\sigma$} if there exists a function $\beta\colon G\to \T$ such that $\beta(e) = 1$ and \[\sigma'(r,s) = \beta(r)\beta(s)\overline{\beta(rs)} \sigma(r,s)\quad \text{ for all $r,s \in G$}.\] 
The (left) \emph{regular $\sigma$-projective representation} $\lambda_\sigma\colon G\to B(\ell^2(G))$ is defined by
\[
\lambda_\sigma(g)\xi(h)=\sigma(g,g^{-1}h)\xi(g^{-1}h)
\]
for $g, h\in G$ and $\xi\in \ell^2(G)$. 
The \emph{reduced twisted group $C^*$-algebra $C^*_r(G,\sigma)$} is the $C^*$-subalgebra of $B(\ell^2(G))$ generated by $\lambda_\sigma(G)$,
and the \emph{twisted group von~Neumann algebra $\L(G,\sigma)$} is the von~Neumann subalgebra of $B(\ell^2(G))$ generated by $\lambda_\sigma(G)$. The unit in both these algebras is the identity operator on $\ell^2(G)$.
If $\sigma'$ is a two-cocycle on $G$ which is similar to $\sigma$ via a map $\beta\colon G\to \T$, then, as is well-known and easy to check, $C_r^*(G, \sigma')$ (resp.~$\L(G,\sigma')$) is $*$-isomorphic to $C_r^*(G,\sigma)$ (resp.~$\L(G,\sigma)$) via a map $\Phi$ sending $\lambda_{\sigma'}(g)$ to $\beta(g)\lambda_\sigma(g)$ for each $g\in G$. 
There is a canonical faithful tracial state $\tau$ on $\L(G, \sigma)$, hence also on $C^*_r(G,\sigma)$, namely the restriction of the vector state on $B(\ell^2(G))$ associated to the characteristic function $\delta_e$ of $\{e\}$ in $G$.

Let the map  $\widetilde{\sigma}\colon G\times G \to \T$ be defined by
\[
\widetilde{\sigma}(h, g)=\sigma(h, g)\overline{\sigma(hgh^{-1},h)},
\]
 Then a simple computation gives that
 \[\lambda_\sigma(h)\lambda_\sigma(g)\lambda_\sigma(h)^* =\widetilde{\sigma}(h,g) \lambda_\sigma(hgh^{-1})\] 
  for all $h,g\in G$.
 Some further computations using the cocycle identity give that for all $r,s, t\in G$ we have:
\begin{gather}
\widetilde{\sigma}(rs,t)=\widetilde{\sigma}(r,sts^{-1})\widetilde{\sigma}(s,t) \label{left-product} \\
\widetilde{\sigma}(r,st)=\overline{\sigma(s,t)}\sigma(rsr^{-1},rtr^{-1})\widetilde{\sigma}(r,s)\widetilde{\sigma}(r,t) \label{right-product}
\end{gather}
These identities are similar to those  stated in \cite[(1.1)-(1.2)]{PR}, but beware that  
our definition of $\widetilde{\sigma}$ is different from theirs:   
if $\widetilde{\sigma}_\textup{PR}$ denotes this function from \cite{PR},
then $\widetilde{\sigma}(r,s)=\widetilde{\sigma}_\textup{PR}(s,r^{-1})$. From \eqref{left-product} and \eqref{right-product} we immediately get
\begin{gather}
\widetilde{\sigma}(rs,t)=\widetilde{\sigma}(r,t)\widetilde{\sigma}(s,t) \quad  \text{ whenever $st =ts$},\label{left-product2} \\
\quad \quad \quad \widetilde{\sigma}(r,st)=\widetilde{\sigma}(r,s)\widetilde{\sigma}(r,t) \quad \text{ whenever $rs =sr$ and $rt = tr$}. \label{right-product2}
\end{gather}

As usual, when $\sigma$ is trivial (i.e., $\sigma(g,h) =1$ for all $g,h\in G$), we skip it in our notation and from our terminology. So, for example, $\L(G)$ denotes the group von Neumann algebra of $G$.
As is well-known, $\L(G)$ is a factor if and only if $G$ is icc, i.e., every nontrivial conjugacy class in $G$ is infinite.
The twisted version of this fact, which is due to Kleppner \cite{Kleppner}, is as follows. 
An element $g\in G$ is called \emph{$\sigma$-regular} if $\sigma(g,h)=\sigma(h,g)$ whenever $h \in G$ and $gh=hg$. As $\sigma$-regularity is a property of conjugacy classes, it makes sense to  
say that $(G,\sigma)$ satisfies \emph{Kleppner's condition} if there is no nontrivial finite $\sigma$-regular conjugacy class in $G$. As shown in \cite{Kleppner}, $\L(G, \sigma)$ is a factor if and only if $(G,\sigma)$ satisfies Kleppner's condition. 
It can also be shown that this is equivalent to $C^*_r(G,\sigma)$ having a trivial center
(see \cite[Theorem~2.7]{primeness}).

The pair $(G,\sigma)$ is said to be \emph{$C^*$-simple} (resp.~to have \emph{the unique trace property}) when $C^*_r(G,\sigma)$ is simple (resp.~has a unique tracial state). Any pair satisfying one of these conditions must necessarily satisfy Kleppner's condition, but the converse implications do not always hold. If $G$ is $C^*$-simple, then, as shown in \cite{BK}, $(G, \sigma)$ is $C^*$-simple and has the unique trace property. We refer to \cite{SUT} for other examples of pairs satisfying one or both of these conditions. 

A large class of amenable groups for which $C^*$-simplicity (resp.~the unique trace property) of $(G,\sigma)$ is known to be equivalent to Kleppner's condition being satisfied is the class of FC-hypercentral groups, i.e., groups having no icc quotient group other than the trivial one, cf.~\cite{FCH} and references therein. This class contains all abelian groups, all FC-groups, all finitely generated groups having polynomial growth and, more generally, all virtually nilpotent groups. Countable FC-hypercentral groups have recently been characterized as groups with the Choquet-Deny property \cite{FHTF}, and as strongly amenable groups \cite{FTF}. 

\section{The relative Kleppner condition} \label{rKc}

Let $H$ be a subgroup of $G$. We will often write $H\leq G$ to indicate this.
The \emph{$H$-conjugacy class} of $g$ in $G$ is defined as \[g^H:=\{ hgh^{-1} : h\in H\}.\]
Thus $g^H$ is the orbit of $g$ under the action of $H$ on $G$ by conjugation, and we have that $\lvert g^H\rvert = \big[H: C_H(g)\big]$, where $C_H(g) := \{ h \in H : hg = gh\}$ is the centralizer of $g$ in $H$. 

We recall that $G$ is said to be an \emph{icc group relatively to $H$} if every nontrivial $H$-conjugacy class in $G$  is infinite. 
We will denote the centralizer of $H$ in $G$ by $C_G(H)$. 
Thus, $C_G(H)=\{ g\in G : \lvert g^H\rvert =1 \}$.
Similarly, we define the \emph{FC-centralizer of $H$ in $G$} as the subgroup of $G$ given by  
\[
FC_G(H)=\{ g\in G : \lvert g^H\rvert <\infty \}.
\]
Clearly, $C_G(H) \subseteq FC_G(H)$, and saying that $G$ is icc relatively to $H$ means precisely that $FC_G(H)$ is trivial. 

\begin{example}
Let $H, K$ be groups and assume that $\alpha\colon K \to \operatorname{Aut}(H)$ is an action of $K$ on $H$ by automorphisms. Let $G= H \rtimes K$ denote the associated semidirect product. As usual, we consider $H$ as a normal subgroup of $G$ and $K$ as a subgroup of $G$, so that $H\cap K = \{e\}$ and $\alpha_k(h)= k h k ^{-1}$ for all $ h\in H$, $k \in K$. Then we have that $G$ is icc relatively to $H$ if and only if $H$ is icc and $\alpha$ is outer, i.e., $\alpha_k$ is an outer automorphism of $H$ for each $k\in K\setminus \{e\}$. 
Indeed, if $H$ is icc and $\alpha$ is outer, then using \cite[Lemma 3.4]{Bedos91} it is not difficult to deduce that  $G$ is icc relatively to $H$. Alternatively, one may check that $C_G(H)$ is trivial and apply Proposition \ref{icc-rel}. The converse implication is straightforward.
On the other hand,  $G$ is icc relatively to $K$ if and only if $K$ is icc and $\{\alpha_k(h): k \in K\}$ is infinite for every $h\in H\setminus \{e\}$. The reader should have no trouble in verifying 
this assertion.
\end{example}
 
 \begin{lemma}\label{icc-normal}
Assume $H$ is icc and normal in $G$. Then $FC_G(H) = C_G(H)$. 
\end{lemma}
\begin{proof}
Since $H$ is icc, we have  that $FC_G(H)\cap H=\{e\}$. Let $g\in FC_G(H)$ and $k \in g^H$. Then $k^H= g^H$, so $k\in FC_G(H)$. Hence $g^{-1}k \in FC_G(H)$. 
But, since $H$ is normal, $g^{-1}k \in H$. Thus, $g^{-1}k = e$, i.e., $k=g$. This means that $\lvert g^H\rvert = 1$, i.e., $g\in C_G(H)$. 
 \end{proof}
Using this lemma, we immediately get:
 \begin{proposition}\label{icc-rel}
 If $H$ is normal in $G$, then $G$ is icc relatively to $H$ if and only if $H$ is icc and $C_G(H)$ is trivial. 
 \end{proposition}

The condition $FC_G(H)= C_G(H)$ is trivially satisfied when 
$H$ is contained in the center of $G$.
It is also satisfied in some other situations.
\begin{example}\label{R-g}
Let $H\leq G$ and assume that $H$ has no nontrivial finite quotient, or that $H$ is an $R$-group (meaning that if $h_1, h_2 \in H$ are such that $h_1^n=h_2^n$ for some $n\in \N$, then $h_1=h_2$). Then $FC_G(H)=C_G(H)$. 

Indeed, suppose (for contradiction) that $FC_G(H)\neq  C_G(H)$. Then there exists $g\in G\setminus \{e\}$ such that
\begin{equation}\label{finiteCC} 
1<\lvert g^H\rvert=[H:C_H(g)]<\infty.
\end{equation}
Thus,  $H$ has a proper subgroup of finite index, and therefore a normal proper subgroup of finite index (just take the normal core). Hence $H$ has a nontrivial finite quotient.
Moreover, \eqref{finiteCC} gives that there must exist some $h\in H\setminus \{e\}$ such that $hgh^{-1}\neq g$. 
It also implies that we must have $h^ngh^{-n}=g$ for some $n\in \N$,
that is $h^n=gh^ng^{-1}=(ghg^{-1})^n$. So $H$ is not an R-group.

 All torsion-free, locally nilpotent groups are known to be $R$-groups (see for example \cite[2.1.2]{LR}). It is also known that every group having a bi-invariant total order is an $R$-group (see e.g.~\cite[Lemma 7.7]{KT}). Combining this with \cite[Theorem 7.8]{KT} one deduces that  all pure braid groups are  also $R$-groups.
 
\end{example}

The following definition was given in \cite{SUT} in the case where $H$ is normal, but it makes sense without this assumption.
\begin{definition}
An element $g\in G$ is said to be \emph{$\sigma$-regular w.r.t.~$H$} if $\sigma(g,h)=\sigma(h,g)$ whenever $h\in H$ and $gh=hg$.
\end{definition}
In other words, $g\in G$ is $\sigma$-regular w.r.t.~$H$ if and only if $\widetilde{\sigma}(h,g) =1$ (resp.~$\widetilde{\sigma}(g,h) =1$)  whenever $h\in H$ and $gh=hg$. 
It follows from the argument given in the proof of \cite[Lemma~4.4]{SUT} that 
$\sigma$-regularity w.r.t.~$H$ is a property of $H$-conjugacy classes.
Also, if $\sigma'$ is a two-cocycle on $G$ which is similar to $\sigma$, then it is easy to check that $g\in G$ is $\sigma$-regular w.r.t.~$H$ if and only if it is $\sigma'$-regular w.r.t.~$H$. The twisted analogue of $G$ being icc relatively to $H$ is as follows.
\begin{definition}\label{rel K}
Let $H\leq G$.Then $(H \leq G, \sigma)$ is said to satisfy the \emph{relative Kleppner condition} if every nontrivial $H$-conjugacy class in $G$ that is $\sigma$-regular w.r.t.~$H$ is infinite.
\end{definition}
Clearly, this property depends on $\sigma$ only up to similarity, and   
$(H \leq G, \sigma)$ automatically satisfies the relative Kleppner condition whenever $G$ is icc relatively to $H$. Also, if $(H \leq G, \sigma)$
satisfies the relative Kleppner condition, then both $(H, \sigma)$ and $(G, \sigma)$ satisfy Kleppner's condition. The converse implication does not necessarily hold, even if $H$ is normal in $G$ and $\sigma$ is trivial. Indeed, if $G= H\times K$ where both $H$ and $K$ are icc groups and $K$ is nontrivial, then $G$ is icc, but not icc relatively to $H$. 

\begin{remark}
The reader should be aware of the discrepancy between Definition \ref{rel K} and 
the relative Kleppner condition introduced in 
\cite[Definition~4.5]{SUT}, where only $H$-conjugacy classes in $G\setminus H$ are involved;
the difference is whether or not the assumption that $(H,\sigma)$ satisfies Kleppner's condition
is part of the definition: in this paper, it is.
\end{remark}

\begin{remark} 
If $H \leq G$ is such that $FC_G(H)=C_G(H)$, 
then $(H \leq G, \sigma)$ satisfies the relative Kleppner condition if and only if  
for every $g\in C_G(H)\setminus\{e\}$, there exists some $h\in H$ such that $\sigma(g, h) \neq \sigma(h, g)$. 
\end{remark} 

 When $H$ is a normal subgroup of $G$, one may wonder if the relative Kleppner property for $(H \leq G, \sigma)$ can be characterized in a way similar to the one obtained in Proposition \ref{icc-rel} when $\sigma$ is trivial.  
 This is a nontrivial problem in general, but we will provide a positive answer for a large class of normal subgroups in Theorem \ref{sigma-centralizer2}. 

We define
\begin{align*}
C_G^\sigma(H) :&= \{ g \in G : \lvert g^H\rvert=1 \text{ and  $g$ is $\sigma$-regular w.r.t.~$H$} \}\\
&= \{ g\in C_G(H) : \widetilde\sigma(g, h) =1 \, \text{ for all } h\in H\}.
\end{align*}
Using  \eqref{left-product2}, one readily checks that $C_G^\sigma(H)$ is a subgroup of $C_G(H)$. Moreover, if $H$ is  normal in $G$, then so is $C_G^\sigma(H)$.
Next, we define
\begin{align*}
FC_G^\sigma(H) :&= \{ g\in G : \lvert g^H\rvert < \infty \text{ and $g$ is $\sigma$-regular w.r.t.~$H$} \}\\
&= \{ g\in FC_G(H) : \text{$g$ is $\sigma$-regular w.r.t.~$H$} \}.
\end{align*}
Clearly, $C_G^\sigma(H) \subseteq FC_G^\sigma(H)$,
and  $(H \leq G, \sigma)$ satisfies the relative Kleppner condition if and only if $FC_G^\sigma(H)$ is trivial.
If $\sigma$ is nontrivial, then $FC_G^\sigma(H)$ is not necessarily a subgroup of $G$. However, the following two properties hold.
\begin{lemma} \label{sigma-inverse}
$FC_G^\sigma(H)$ is closed under the inverse operation.
\end{lemma}
\begin{proof}
Let $g \in FC_G^\sigma(H)$. Then $g\in FC_G(H)$, so $g^{-1} \in FC_G(H)$. 
Assume that $h\in H$ commutes with $g^{-1}$. 
Then $h$ also commutes with $g$, so $\widetilde{\sigma}(g,h)=1$.
Therefore, using \eqref{left-product2}, we obtain
\[1 =  \widetilde{\sigma}(g^{-1}g, h) = \widetilde{\sigma}(g^{-1}, h)\widetilde{\sigma}(g,h) = \widetilde{\sigma}(g^{-1}, h).\]
This shows that $g^{-1}$ is $\sigma$-regular w.r.t.~$H$. Thus, $g^{-1} \in FC_G^\sigma(H)$.
\end{proof}

\begin{lemma}\label{sigma-product}
Let $r \in \N$ and suppose that $g_1,\dotsc,g_r\in FC_G^\sigma(H)$.
Then there exists an $n\in \N$ such that $(g_1\dotsm g_r)^{mn}\in FC_G^\sigma(H)$ for all $m\in \N$.
\end{lemma}

\begin{proof}
Let $g,k\in FC_G^\sigma(H)$. Since $FC_G(H)$ is a group,  we have that $(gk)^n \in FC_G(H)$
for every $n\geq 1$.
Set $R:=C_H(g)\cap C_H(k)=C_H(\{g,k\})$, and for each $n\geq 1$, set $Q_n:=C_H((gk)^n)$.
	
Clearly, $R$ is contained in $Q_n$ for every $n$, and it has finite index in $H$ since it is an intersection of finite index subgroups.
Moreover, if $m$ divides $n$, then $Q_m\subseteq Q_n$, so
\[
[Q_1 : R] \leq [Q_m : R] \leq [Q_n : R] \leq [H : R] < \infty.
\]
It follows from \eqref{left-product2} that
\[
\widetilde{\sigma}(xy,(gk)^n)=\widetilde{\sigma}(x,(gk)^n)\, \widetilde{\sigma}(y,(gk)^n)
\]
for all $x\in G$ and $y\in Q_n$. In particular, for every $n\geq 1$, the map $\nu_n\colon Q_n\to\T$ given by 
$x\mapsto\widetilde{\sigma}(x,(gk)^n)$ is a homomorphism. 
Define
\[
\begin{split}
		\kappa &:= \max\{ [Q_n : R] : n\geq 1\}, \\
		\ell &:= \min\{ n : [Q_n : R]=\kappa \}.
	\end{split}
\]
We note that $[Q_n : R]=\kappa$ and $Q_n = Q_\ell$ whenever $n$ is a multiple of $\ell$, 
since we then have that
\[ \kappa\geq	[Q_n : R] = [Q_n:Q_\ell]\,  [Q_\ell:R]  = [Q_n:Q_\ell]\, \kappa\, \geq \,\kappa.\]
By applying \eqref{right-product2} repeatedly, we get that the homomorphism $\nu_\ell\colon Q_\ell\to\T$ maps $R$ to $\{1\}$. Indeed, for every $y \in R$, we have $\widetilde{\sigma}(y,g)= \widetilde{\sigma}(y,k)=1$, hence
\begin{align*}
\widetilde{\sigma}(y, gk) &= 
\widetilde{\sigma}(y,g)\widetilde{\sigma}(y,k) = 1,
\end{align*}
and therefore
\begin{align*} 
\nu_\ell(y) = \widetilde{\sigma}(y, (gk)^\ell) &= 
(\widetilde{\sigma}(y,gk))^\ell = 1.
\end{align*}
Since $[Q_\ell : R]=\kappa$, this means that the image of $\nu_\ell$ 
is a finite subgroup of $\T$ with at most $\kappa$ elements. 
Hence, for each $m\geq 1$ and $x\in Q_{m\kappa\ell}=Q_\ell$,  applying again \eqref{right-product2} repeatedly,  we get that
\[
\widetilde{\sigma}(x,(gk)^{m\kappa\ell})=\widetilde{\sigma}(x,(gk)^\ell)^{m\kappa}= (\nu_\ell(x))^{m\kappa} = (\nu_\ell(x)^\kappa)^m=1.
\]
In other words, $\sigma(x,(gk)^{m\kappa\ell}) =\sigma((gk)^{m\kappa\ell}, x) $ for every $x\in H$ which commutes with $(gk)^{m\kappa\ell}$, i.e., $(gk)^{m\kappa\ell}$ is $\sigma$-regular w.r.t~$H$. Thus, $(gk)^{m\kappa\ell} \in FC^\sigma_G(H)$ for all $m\geq 1$. 

This shows that the conclusion holds when $r=2$. It also holds if $r=1$ (by setting $g_2 := e$). The proof when $r>2$ proceeds in essentially the same way, now with $R:= \cap_{j=1}^r C_H(g_j)$ and $Q_n= C_H((g_1\cdots g_r)^n)$.  
\end{proof}

We recall that the FC-center $FC(G)$ of $G$ consists of the elements of $G$ having a finite $G$-conjugacy class. In other words, $FC(G) = FC_G(G)$. 
The equivalence between conditions (i) and (ii) in the following proposition is pointed out by Connell in \cite[p.~675]{Connell}.

\begin{proposition}\label{prime-def}
The following conditions are equivalent:
\begin{itemize}
	\item[(i)] $G$ has no finite normal subgroup except $\{e\}$.
	\item[(ii)] The FC-center of $G$   is a torsion-free group.
	\item[(iii)] The FC-center of $G$ is a torsion-free abelian group.
\end{itemize}
\end{proposition}
\begin{proof}
The equivalences betweeen (i), (ii), and (iii) can be deduced from e.g.~\cite[Theorem~4.32]{Robinson}. Indeed, assume that there exists some $r\in FC(G)\setminus \{e\}$ having torsion.  Then the normal subgroup $N$ of $G$ generated by the $G$-conjugacy class $r^G$  is contained in $FC(G)$. Hence, $N$ is locally finite by  \cite[Theorem~4.32, ii)]{Robinson}. Since $N$ is finitely generated, this means that it is finite. This shows that (i) $\Rightarrow $ (ii). Moreover, if $r, s \in FC(G)$, then \cite[Theorem~4.32, ii)]{Robinson} also gives that $rsr^{-1}s^{-1}$ is a torsion element of $FC(G)$. So if (ii) holds, then we get that $rs=sr$. Thus, (ii) $\Rightarrow $ (iii), and the converse implication is trivial. Finally, assume that $G$ has a nontrivial finite normal subgroup $N$. Then $N$ is contained $FC(G)$. Since $N$ is finite, we get that $FC(G)$ contains nontrivial torsion elements. This shows that (ii) $\Rightarrow $ (i).
 \end{proof}
Connell calls the group $G$ \emph{prime}  whenever  condition (i) (or (ii)) holds, and shows in \cite[Theorem 8]{Connell} that 
$G$ is prime if and only if the group algebra $\C[G]$ is prime.
We will adopt his terminology.
\begin{definition} \label{primedef} The group $G$ is said to be \emph{prime} when it satisfies any of the equivalent conditions in Proposition \ref{prime-def}. 
\end{definition}
Clearly, every icc group and every torsion-free group is prime. An $R$-group (cf.~Example \ref{R-g}) is prime if and only if its center is torsion-free. As shown by Gelander and Glasner \cite[Theorem 1.15]{GG}, a countable non-elementary convergence group (e.g., a subgroup of a Gromov hyperbolic group, or a Kleinian group) is prime if and only if it admits a faithful primitive action on a set.

Our next result is a twisted analogue of Lemma \ref{icc-normal} in the case where $H$ is prime.

\begin{lemma}\label{sigma-centralizer}
Let $H$ be a normal subgroup of $G$.
Assume that $H$ is prime and that $(H,\sigma)$ satisfies Kleppner's condition.
Then
\[
C_G^\sigma(H) = FC_G^\sigma(H).
\]
\end{lemma}

\begin{proof}
Suppose that $g\in FC_G^\sigma(H)$ and $k\in g^H$.
Then $k \in FC_G^\sigma(H)$, and $g^{-1} \in FC_G^\sigma(H)$ by Lemma \ref{sigma-inverse}. So Lemma~\ref{sigma-product} gives that $(g^{-1}k)^n$ is $\sigma$-regular w.r.t.~$H$ for some $n\geq 1$.
Moreover, $g^{-1}k \in H$ (since $H$ is normal and $k = hgh^{-1}$ for some $h\in H$).
As $(g^{-1}k)^H$ is finite, we get that $g^{-1}k\in FC(H)$, hence that $(g^{-1}k)^n \in FC(H)$.
 Since $(H,\sigma)$ satisfies Kleppner's condition, this means that $(g^{-1}k)^n=e$.
As $H$ is assumed to be prime, $FC(H)$ is torsion-free, so we must have $g^{-1}k=e$, that is, $k=g$. This shows that $\lvert g^H\rvert=1 $. Hence, $FC_G^\sigma(H) =C_G^\sigma(H)$.
\end{proof}
\begin{remark} 
When 
$C_G^\sigma(H) = FC_G^\sigma(H)$,
e.g., when $H$ satisfies the assumptions in Lemma \ref{sigma-centralizer}, 
then by modifying the argument of Kleppner in \cite[Lemma~4]{Kleppner}, one may deduce that  $C^*_r(H,\sigma)' \cap C^*_r(G,\sigma) = C^*_r(C_G^\sigma(H),\sigma)$ and $\L(H,\sigma)' \cap \L(G,\sigma) = \L(C_G^\sigma(H),\sigma)$. 

\end{remark}

As an immediate consequence of Lemma \ref{sigma-centralizer}, we get:
\begin{theorem} \label{sigma-centralizer2}
Let $H$ be a normal subgroup of $G$ and assume that $H$ is prime.
Then $(H\leq G,\sigma)$ satisfies the relative Kleppner condition  if and only if  $(H,\sigma)$ satisfies Kleppner's condition and $C_G^\sigma(H)$ is trivial. 
\end{theorem}

\section{Irreducibility and primeness} \label{irred-prim}
Let  $H$ be a subgroup of $G$. We will denote the restriction of $\sigma$ to $H\times H$ by the same symbol $\sigma$. 
 It is well-known that $C^*_r(H,\sigma)$ (resp.~$\L(H,\sigma)$) may be identified with the $C^*$-subalgebra of $C_r^*(G,\sigma)$ (resp.~the von Neumann subalgebra of $\L(G,\sigma)$) generated by $\lambda_\sigma(H)$, cf.~\cite[Subsection 6.46]{ZM}.
 
The following two lemmas are straightforward generalizations of \cite[Lemmas~2.2 and~2.3]{primeness} (the assumption that $H$ is normal is not used in the proof of these two results).
\begin{lemma}\label{center}
Let $H\leq G$,  $T \in \L(G,\sigma)$ and set $f_T=T\delta_e \in \ell^2(G)$. 
Then the following conditions are equivalent:
\begin{itemize}
\item[(i)] $T$ belongs to $\L(H,\sigma)'\cap \L(G,\sigma)$.
\item[(ii)] $f_T(hgh^{-1}) = \widetilde{\sigma}(h,g) f_T(g)$ for all $h\in H$ and $g\in G$.
\end{itemize}
Moreover, $f_T$ can be nonzero only on the finite $H$-conjugacy classes.
\end{lemma}

\begin{lemma}\label{function-exists}
Let $H\leq G$ and $C$ be an $H$-conjugacy class in $G$.
Then the following conditions are equivalent:
\begin{itemize}
\item[(i)] $C$ is $\sigma$-regular w.r.t.~$H$.
\item[(ii)] There is a function $f\colon G\to \C$ satisfying:
\begin{itemize}
\item[1.] $f(g)\neq 0$ for all $g\in C$.
\item[2.] $f(hgh^{-1}) = \widetilde{\sigma}(h,g)
f(g)$ for all $g\in C$ and all $h\in H$.
\end{itemize}
\end{itemize}
Moreover, $f$ can be chosen in $\ell^2(G)$ if and only if $C$ is finite.
\end{lemma}

The twisted analogue of \cite[Proposition 5.1]{Rordam} is the following:
\begin{proposition}\label{kleppner-commutant}
Let $H\leq G$. Then the following assertions are equivalent:
\begin{itemize}
\item[(i)] $(H \leq G, \sigma)$ satisfies the relative Kleppner condition.

\item[(ii)] $\L(H,\sigma)'\cap \,\L(G,\sigma)=\C1$, i.e., the  inclusion $\L(H,\sigma)\subseteq \L(G,\sigma)$
 is irreducible.
\item[(iii)] $C^*_r(H,\sigma)'\cap C^*_r(G,\sigma)=\C1$,
i.e., the  inclusion $C_r^*(H,\sigma)\subseteq C_r^*(G,\sigma)$ is irreducible.
\end{itemize}

\smallskip \noindent Moreover, if $G$ is countable, then $\L(H, \sigma) \subseteq \L(G, \sigma)$ is $C^*$-irreducible if and only if $(H \leq G, \sigma)$ satisfies the relative Kleppner condition and $[G:H] < \infty$.
\end{proposition}

\begin{proof}
(i) $\Longrightarrow$ (ii) follows readily from Lemmas~\ref{center} and~\ref{function-exists}. 

(ii) $\Longrightarrow$ (iii) is obvious.

(iii) $\Longrightarrow$ (i):
If there is a finite nontrivial $H$-conjugacy class $C$ in $G$ that is $\sigma$-regular w.r.t.~$H$,
then Lemma~\ref{function-exists} ensures that there exists a function $f\colon C\to\C$ such that $\sum_{g\in C}f(g)\lambda_\sigma(g)$ is nonscalar and belongs to 
$C^*_r(H,\sigma)'\cap C^*_r(G,\sigma)$.

Finally, assume that $G$ is countable. In view of the equivalence of (i) and (ii), we have to show that 
$\L(H, \sigma) \subseteq \L(G, \sigma)$ is $C^*$-irreducible if and only if $\L(H,\sigma)\subseteq \L(G,\sigma)$
 is irreducible and $[G:H] < \infty$. We may then assume that both $\L(H, \sigma)$ and $\L(G, \sigma)$ are separable II$_1$-factors. Then the desired equivalence follows from \cite[Theorem 4.4]{Rordam} (which R\o rdam attributes to Pop and Popa), taking into account that it is well-known that the Jones index $[\L(G, \sigma) : \L(H, \sigma)]$ is equal to $[G:H]$.
\end{proof}

For  regular irreducible inclusions of twisted group von~Neumann algebras, intermediate von Neumann algebras have a particularly simple description:
\begin{proposition}\label{galois-vN}
Assume that $H$ is normal in $G$ and  $(H \leq G, \sigma)$ satisfies the relative Kleppner condition, i.e., the inclusion $\L(H,\sigma)\subseteq \L(G,\sigma)$ is irreducible. Then the map $\Gamma \mapsto \L(\Gamma,\sigma)$ gives a
bijective correspondence between groups $\Gamma$ satisfying $H \leq \Gamma \leq G$, and  von Neumann algebras $N$ satisfying $\L(H,\sigma) \subseteq N\subseteq \L(G,\sigma)$. 
\end{proposition}

\begin{proof} 
This result could be proven by using results from \cite{ILP} or \cite{CS}, or by decomposing $\L(G, \sigma)$ as a twisted crossed product and applying \cite[Lemma 3.3]{Jiang}. For the ease of the reader, we sketch a direct proof, close in spirit to the proof of \cite[Corollary 4]{Choda}. 

Clearly, if $H \leq \Gamma \leq G$, then $N_\Gamma:=  \L(\Gamma,\sigma)$ is a subfactor of  $\L(G,\sigma)$ containing $\L(H,\sigma)$. Conversely, let $N$ be a von Neumann algebra satisfying
$\L(H,\sigma) \subseteq N\subseteq \L(G,\sigma)$, and let $E_\tau^N$ denote the faithful normal conditional expectation from $\L(G,\sigma)$ onto $N$ associated to $\tau$.
 We note that for any $g\in G$, we have that $\lambda_\sigma(g)^*E_\tau^N(\lambda_\sigma(g)) \in \C 1$.
 
Indeed, let $x \in \L(H, \sigma)$. Then  $\alpha_g(x): = \lambda_\sigma(g) x \lambda_\sigma(g)^* \in \L(H, \sigma)$ (because $H$ is normal in G), so we get
\[\alpha_g(x) E_N^\tau(\lambda_\sigma(g)) = E_N^\tau(\alpha_g(x)\lambda_\sigma(g))=E_N^\tau(\lambda_\sigma(g)x) = E_N^\tau(\lambda_\sigma(g)) x.\]
 This implies that  \[x \, \lambda_\sigma(g)^* E_N^\tau(\lambda_\sigma(g)) = \lambda_\sigma(g)^*E_N^\tau(\lambda_\sigma(g))\, x.\]  
Thus we get that  $ \lambda_\sigma(g)^* E_N^\tau(\lambda_\sigma(g)) \in \L(H,\sigma)'\cap \,\L(G,\sigma)=\C1$, as asserted.

Now, set $\Gamma_N:= \{ g\in G : \lambda_\sigma(g) \in N\}$. Then $\Gamma_N$ is easily seen to be a subgroup of $G$ containing $H$, and the observation above gives that
$E_\tau^N(\lambda_\sigma(g)) = 0$ whenever $g\not \in \Gamma_N$. 
This readily  implies that $E_\tau^N$ maps $\L(G, \sigma)$ into $\L(\Gamma_N, \sigma)$, hence that 
\[ N= E_N^\tau(\L(G, \sigma)) \subseteq \L(\Gamma_N, \sigma) \subseteq N. \] Thus $N=\L(\Gamma_N, \sigma) =N_{\Gamma_N}$.  This shows that the map  $\Gamma \mapsto N_\Gamma$ is surjective. 
To see that it is injective, it suffices to show that $\Gamma= \Gamma_{N_\Gamma}$ whenever  $H \leq \Gamma \leq G$. By definition, the inclusion $\subseteq $ holds. 
On the other hand, let $g \in \Gamma_{N_\Gamma}$, i.e., $g\in G$ and $\lambda_\sigma(g)  \in N_\Gamma = \L(\Gamma, \sigma)$. Then we have \[\delta_g = \lambda_\sigma(g)\delta_e \in \L(\Gamma, \sigma)\delta_e \subseteq \{\xi \in \ell^2(G) : \xi(s) = 0 \text{ for all } s \in G\setminus \Gamma\},\] which is possible only  if $g\in \Gamma$. 
  \end{proof}
  \begin{remark} The assumption in Proposition \ref{galois-vN} that $H$ is normal can not be removed. 
   This can be seen by using an observation due to Jiang in \cite[Section 4]{Jiang}.  Let $K$ be an infinite simple group, and set $G:= K\times K$ and $H=\{(k, k) : k\in K\}$. Then $H$ is a maximal subgroup of $G$, and as shown by Jiang, $\L(H)$ is not maximal in $\L(G)$, i.e., there exists a von Neumann subalgebra $N$  of $\L(G)$ containing $\L(H)$ such that $\L(H)\neq N \neq \L(G)$. Since $K$ is icc, it is easy to verify that $G$ is icc relatively to $H$.  
  \end{remark}
  
Recall that a $C^*$-algebra is said to be \emph{prime} if nonzero ideals have nonzero intersection.
As shown in \cite[Theorem~2.7]{primeness}, $(G,\sigma)$ satisfies Kleppner's condition if and only if $C^*_r(G,\sigma)$ is prime, if and only if $C^*_r(G,\sigma)$ has trivial center.
In this connection, we note that if  $B\subseteq A$ is a unital inclusion of $C^*$-algebras, then the following conditions are equivalent:
\begin{itemize}
	\item[(i)] The inclusion $B\subseteq A$ is irreducible.
	\item[(ii)] Every $C^*$-algebra $C$ satisfying $B\subseteq C\subseteq A$ has trivial center.
\end{itemize}
Indeed, if (i) holds and $C$ is an intermediate $C^*$-algebra, then $C' \cap C\subseteq B' \cap A=\C1$.
Conversely, if (i) does not hold, and we choose a nonscalar $a\in B' \cap A$, then the $C^*$-algebra $C$ generated by $B$ and $a$ has a nontrivial center.

Using Proposition~\ref{kleppner-commutant} we get the following generalization:

\begin{corollary} The following conditions are equivalent:
\begin{itemize}
\item[(i)] $(H \leq G, \sigma)$ satisfies the relative Kleppner condition.

\item[(ii)] Every $C^*$-algebra $A$ satisfying $C^*_r(H,\sigma)\subseteq A\subseteq C^*_r(G,\sigma)$ is prime.

\item[(iii)] Every $C^*$-algebra $A$ satisfying $C^*_r(H,\sigma)\subseteq A\subseteq C^*_r(G,\sigma)$ has trivial center.
\end{itemize}
\end{corollary}

\begin{proof}
(i)~$\Longrightarrow$~(ii): Assume (i) holds. Then Proposition~\ref{kleppner-commutant} gives that $\L(H,\sigma)\subseteq \L(G,\sigma)$ is irreducible. Now, if $A$ is a $C^*$-algebra satisfying $C^*_r(H,\sigma)\subseteq A\subseteq C^*_r(G,\sigma)$, then we get $\L(H,\sigma)\subseteq A'' \subseteq \L(G,\sigma)$, so $A''$ is a factor; as is well-known, this implies that $A$ is prime. Hence, (ii) holds. 

(ii)~$\Longrightarrow$~(iii): Since every prime $C^*$-algebra has trivial center, this is clear.

(iii)~$\Longrightarrow$~(i):  Assume (iii) holds. Then, as observed above,
$C^*_r(H,\sigma)\subseteq C^*_r(G,\sigma)$ is irreducible. Hence, (i) holds by Proposition~\ref{kleppner-commutant}.
\end{proof}

\section{$C^*$-irreducibility and reduced twisted $C^*$-crossed products} \label{irred-crpro}

In this section, we assume that $(A, G, \alpha,\sigma)$ is a  unital, discrete, twisted 
C$^*$-dynamical system, i.e., $A$ is a $C^*$-algebra with unit $1$, 
$G$ is a discrete group with identity $e$ and $(\alpha,\sigma)$ is a  twisted
action of $G$ on $A$, that is, $\alpha$ is a map from $G$ into the group of $*$-automorphisms of $A$,
and  $\sigma$ is a map from $G \times G$ into the unitary group of $A$, satisfying
\[
\begin{gathered}
\alpha_g \circ \alpha_h = \operatorname{Ad}(\sigma(g, h)) \circ  \alpha_{gh}, \\
\sigma(g,h) \sigma(gh,k) = \alpha_g(\sigma(h,k)) \sigma(g,hk), \\
\sigma(g,e)  = \sigma(e,g) = 1,
\end{gathered}
\]
for all $g,h,k \in G$, where $\operatorname{Ad}(v)$ denotes the (inner) $*$-automorphism  of $A$ implemented by a unitary $v$ in $A$.

For simplicity, we will denote the reduced crossed product associated to $(A,G,\alpha,\sigma)$ by $A\rtimes_r G$. (If confusion may arise, we will denote it by $A\rtimes_{(\alpha, \sigma),r} G$.)   
As a $C^*$-algebra, $A\rtimes_r G$ is generated  by (a copy of) $A$  and a family $\{u(g) : g\in G\}$ of unitaries satisfying
\[
\alpha_g(a) = u(g)a u(g)^* \quad \text{ and } \quad u(g) u(h) = \sigma(g,h) u(gh)
\]
for all $g,h \in G$ and $a \in A$. Moreover, 
there exists a faithful conditional expectation $E \colon A\rtimes_r G \to A$ such that $E(u(g)) =0$ for  all $g \in G\,,\,  g \neq e\,.$

We recall that $(\alpha, \sigma)$ is said to be \emph{outer} if $\alpha_g$ is outer for each $g\in G\setminus\{e\}$. 
The twisted version of \cite[Theorem~5.8]{Rordam} is as follows:
\begin{theorem}\label{twisted-outer}
The following conditions are equivalent:
\begin{itemize}
\item[(i)] $A\subseteq A\rtimes_r G$ is $C^*$-irreducible.
\item[(ii)] $A$ is simple and $(\alpha, \sigma)$ is outer.
\item[(iii)] $A$ is simple and $A' \cap (A\rtimes_r G) = \C1$.
\end{itemize}
Moreover, if $A$ has the Dixmier property, then any of the above conditions implies that $A\subseteq A\rtimes_r G$ has the relative Dixmier property $($as defined in \cite{Popa}$)$.
\end{theorem}

\begin{proof}
	
(i) $\Longrightarrow$ (iii): Follows from \cite[Remark~3.8]{Rordam}.

(iii) $\Longrightarrow$ (ii): Same argument as in the proof of \cite[Theorem~5.8]{Rordam}.

(ii) $\Longrightarrow$ (i):	Assume that (ii) holds. As shown in \cite[Theorem~3.2]{Bedos91}, Kishimoto's result that $A\rtimes_r G$ is simple remains true for twisted actions. In fact, arguing as in the proof of \cite[Theorem~5.8]{Rordam}, one may show that $A\subseteq A\rtimes_r G$ is $C^*$-irreducible.  Indeed, let $B_0$ denote the $*$-algebra generated by $A$ and $\{u(g):g \in G\}$, i.e., $B_0=\operatorname{Span}\{au(g) : a \in A, g\in G\}$.  Using \cite[Lemma 3.2]{Kish} (or \cite[Lemma 7.1]{OP} if $A$ is separable), and the denseness of $B_0$ in $A\rtimes_r G$, one deduces that  $A \subseteq A\rtimes_r G$ has the pinching property with respect to $E$ (cf.~\cite[Definition~3.13]{Rordam}). The conclusion follows then from  \cite[Proposition~3.15]{Rordam}.
	
The last statement follows from \cite[Corollary~4.1]{Popa}.
\end{proof}

\medskip 
Cameron and Smith's result about intermediate $C^*$-algebras in simple reduced twisted $C^*$-crossed products \cite[Theorem~4.4]{Cameron-Smith} may now be reformulated as follows.
\begin{theorem}\label{Cameron-Smith}
Assume that $A\subseteq A \rtimes_r G$ is $C^*$-irreducible. 
Then there is a bijective correspondence between subgroups $\Gamma$ of $G$ and $C^*$-algebras $B$ satisfying $A\subseteq B\subseteq A\rtimes_r G$, given by
\[
\Gamma \mapsto A\rtimes_r \Gamma.
\]
\end{theorem}
\begin{proof} Since $A$ is simple and  $(\alpha,\sigma)$ is outer by Theorem \ref{twisted-outer}, the conclusion follows from \cite[Theorem~4.4]{Cameron-Smith}.
\end{proof}

Assume that $H$ is a normal subgroup of $G$ and set $K:= G/H$. By restricting $\alpha$ to $H$ and $\sigma$ to $H\times H$, we get a twisted $C^*$-dynamical system  $(A,H,\alpha,\sigma)$. The associated reduced crossed product $A\rtimes_{(\alpha, \sigma),r} H$ may then be identified with the $C^*$-subalgebra of $A\rtimes_{(\alpha, \sigma),r} G$ generated by $A$ and $\{u(h) : h \in H\}$.  We recall from \cite[Theorem~2.1]{Bedos91} that there exists an induced twisted action $(\beta, \omega)$ of $K$ on $A\rtimes_{(\alpha, \sigma),r} H$ such that 
\[ (A\rtimes_{(\alpha, \sigma),r} H) \rtimes_{(\beta, \omega), r} K \, \simeq  \, A\rtimes_{(\alpha, \sigma),r} G\] under a $*$-isomorphism which maps $A\rtimes_{(\alpha, \sigma),r} H$ onto its canonical copy in $A\rtimes_{(\alpha, \sigma),r} G$.

Combining this decomposition result with Theorem \ref{twisted-outer} and Theorem \ref{Cameron-Smith}, we obtain:

 \begin{theorem} \label{outer-quotient}
Let $H$ be a normal subgroup of $G$.
Then the following conditions are equivalent:
\begin{itemize}
\item[(i)] $A\rtimes_r H\subseteq A\rtimes_r G$ is $C^*$-irreducible.
\item[(ii)] $A\rtimes_r H$ is simple and the induced twisted action of $G/H$ on $A\rtimes_r H$ is outer.
\item[(iii)] $A\rtimes_r H$ is simple and $(A\rtimes_r H)' \cap (A\rtimes_r G) = \C1$.
\end{itemize}
Moreover, if $A\rtimes_r H \subseteq A\rtimes_r G$ is $C^*$-irreducible, 
then there is a bijective correspondence between groups $\Gamma$ satisfying $H \leq \Gamma \leq G$, and $C^*$-algebras $B$ satisfying $A\rtimes_r  H \subseteq B\subseteq A\rtimes_r G$, 
given by
\[
\Gamma \mapsto A\rtimes_r \Gamma.
\]
\end{theorem}
In turn, combining this result with \cite[Lemmas 3.3 and 3.4]{Bedos91}, we obtain the following strengthening of \cite[Theorem 3.5]{Bedos91} (which only asserts the simplicity of $A\rtimes_r G$):
\begin{theorem} \label{icc-subgroup}
Let $H$ be a normal subgroup of $G$, and assume that the following conditions are satisfied:
\begin{itemize}
\item $A\rtimes_r H$ is simple. 
\item $A$ has a faithful $G$-invariant state.
\item $H$ is  icc and  $C_G(H)$ is trivial \textup{(}i.e., $G$ is icc relatively to $H$ by Proposition \ref{icc-rel}\textup{)}.
\end{itemize}
Then 
$A\rtimes_r H\subseteq A\rtimes_r G$ is $C^*$-irreducible.
\end{theorem}

\begin{remark} The assumption in Theorem \ref{icc-subgroup} that $A$ has a faithful $G$-invariant state is of a technical nature for the proof of \cite[Lemmas 3.3]{Bedos91} to go through, and it might be redundant. Anyhow, it is for example satisfied when $A$ has a unique faithful tracial state.
\end{remark}

A noteworthy consequence of Theorem \ref{icc-subgroup} is:
\begin{corollary} \label{C*s-sbgp}
Let $H$ be a normal subgroup of $G$ and assume that the following conditions are satisfied:
\begin{itemize}
\item $A$ is $H$-simple \textup{(}i.e., the only $H$-invariant ideals of $A$ are $\{0\}$ and $A$\textup{)}.
\item $A$ has a faithful $G$-invariant state.
\item $H$ is $C^*$-simple and $C_G(H)$ is trivial.
\end{itemize}
Then $A\rtimes_r H\subseteq A\rtimes_r G$ is $C^*$-irreducible.
\end{corollary}
\begin{proof} As $A$ is $H$-simple and $H$ is $C^*$-simple, we get from \cite[Corollary 4.4]{BK} that $A\rtimes_r H$ is simple. As any $C^*$-simple group is icc, we can apply Theorem \ref{icc-subgroup}. \end{proof}
As will be seen in Theorem \ref{FCGH}, if $H$ is normal in $G$, then  $H$ is $C^*$-simple and $C_G(H)$ is trivial if and only if $H\leq G$ is $C^*$-irreducible.

\section{$C^*$-irreducibility and reduced twisted group $C^*$-algebras} \label{irred-rtgc}
 
We recall that $G$ is a discrete group, $H$ is a subgroup of $G$, and $\sigma$ is a two-cocycle on $G$. 

\begin{definition} We will say that \emph{$(H \leq G, \sigma)$ is $C^*$-irreducible} 
if  $C^*_r(H,\sigma)\subseteq C^*_r(G,\sigma)$ is $C^*$-irreducible.
\end{definition}
If $\sigma'$ is a two-cocycle on $G$ which is similar to $\sigma$, then it is clear that the canonical $*$-isomorphism $\Phi\colon C_r^*(G, \sigma') \to C_r^*(G,\sigma)$ maps  $C_r^*(H, \sigma')$ onto $C_r^*(H,\sigma)$, and it therefore follows that  $(H\leq G, \sigma')$ is  $C^*$-irreducible if and only if $(H\leq G, \sigma)$ is  $C^*$-irreducible. 

When $H$ is normal in $G$,  $C^*$-irreducibility of $(H \leq G,\sigma)$ can be characterized as follows: 
\begin{theorem}\label{relative-simplicity}
Assume  $H$ is  normal in $G$. Then the following conditions are equivalent:
\begin{itemize}
	\item[(i)] $(H \leq G,\sigma)$ is $C^*$-irreducible.

 	\item[(ii)] $(H,\sigma)$ is $C^*$-simple and $(H \leq G, \sigma)$ satisfies the relative Kleppner condition.
\end{itemize}

\noindent Both these conditions are satisfied if the following holds:
\begin{itemize}
\item[(iii)] $C^*_r(H,\sigma)\subseteq C^*_r(G,\sigma)$ has the relative Dixmier property.
\end{itemize}

\noindent Moreover, if $(H,\sigma)$ has the unique trace property, then all three conditions are equivalent.
 
 \medskip Finally, if $(H \leq G,\sigma)$ is $C^*$-irreducible, then the map $\Gamma \mapsto C_r^*(\Gamma,\sigma)$ gives a
bijective correspondence between groups $\Gamma$ satisfying $H \leq \Gamma \leq G$, and $C^*$-algebras $B$ satisfying $C_r^*(H,\sigma) \subseteq B\subseteq C_r^*(G,\sigma)$. 
\end{theorem}

\begin{proof}
The equivalence between (i) and (ii) follows from Theorem \ref{outer-quotient} in the case where $A=\C$, in combination with Proposition~\ref{kleppner-commutant}.
 The implication (iii) $\Longrightarrow$ (i) is a consequence of \cite[Proposition~3.12]{Rordam}.
If $(H,\sigma)$ has the unique trace property, then (ii) $\Longrightarrow$ (iii) follows from \cite[Examples~4.3]{Popa}. Indeed, the observation made there also works in the twisted case, by decomposing $C_r^*(G, \sigma)$ as a twisted crossed product $C^*_r(H,\sigma) \rtimes_{(\beta,\omega), r} (G/H)$ and noticing that the twisted action $(\beta, \omega)$ of $G/H$ is then outer by Theorem \ref{outer-quotient}.
The last assertion follows from Theorem \ref{outer-quotient}
\end{proof}

\begin{corollary}\label{FC-normal}
Let $H$ be a normal subgroup of $G$, and assume $H$ is FC-hypercentral or $C^*$-simple.
Then the following conditions are equivalent:
\begin{itemize}
	\item[(i)] $(H \leq G,\sigma)$ is $C^*$-irreducible.
		
\item[(ii)]  $(H \leq G, \sigma)$ satisfies the relative Kleppner condition.
 
\item[(iii)] $C^*_r(H,\sigma)\subseteq C^*_r(G,\sigma)$ has the relative Dixmier property.
\end{itemize}
\end{corollary}
\begin{proof}
As shown in \cite[Theorem 3.1]{FCH}, if $H$ is FC-hypercentral, then $(H, \sigma)$ is $C^*$-simple if and only if $(H, \sigma)$ has the unique trace property, if and only if $(H, \sigma)$ satisfies Kleppner's condition.  
On the other hand, if $H$ is $C^*$-simple, then $(H, \sigma)$ is $C^*$-simple and has the unique trace property, cf.~\cite[Corollaries 4.5 and 5.3]{BK}.  
Hence, Theorem~\ref{relative-simplicity} gives the assertion.
\end{proof}
We will obtain a related result involving the twisted centralizer $C_G^\sigma(H)$ in Corollary \ref{FC-C*-simple-sigma}.
Breuillard, Kalantar, Kennedy and Ozawa have shown in \cite[Theorem 1.4]{BKKO} that if $H$ is normal in $G$, then $G$ is $C^*$-simple if and only if both $H$ and $C_G(H)$ are $C^*$-simple. In general, if $H$ is normal in $G$ and $G$ is $C^*$-simple, then the inclusion $H\leq G$ is not necessarily $C^*$-irreducible (consider for example $G= H\times K$ with $H$ and $K$ both $C^*$-simple). In fact, we have:

\begin{theorem}\label{FCGH}
Assume $H$ is normal in $G$. Then the following conditions are equivalent:
\begin{itemize}

 \item[(i)] $H\leq G$ is $C^*$-irreducible.
 
\item[(ii)] $H$ is $C^*$-simple and $G$ is icc relatively to $H$.

\item[(iii)] $H$ is $C^*$-simple and $C_G(H)$ is trivial.

\item[(iv)] $C^*_r(H)\subseteq C^*_r(G)$ has the relative Dixmier property.

\item[(v)] $H\leq G$ is relatively $C^*$-simple.
\end{itemize}
\end{theorem}

\begin{proof}
 Theorem~\ref{relative-simplicity} gives 
 that (i) is equivalent to (ii).
Since any $C^*$-simple group is icc, the equivalence (ii)~$\Longleftrightarrow$~(iii) follows from Proposition \ref{icc-rel}.
 Taking into account that $H$ has the unique trace property whenever $H$ is $C^*$-simple, cf.~\cite[Theorem 1.3]{BKKO}, the equivalence between (i) and (iv) follows also from Theorem~\ref{relative-simplicity}. 
Finally, the equivalence (iii)~$\Longleftrightarrow$~(v) is a consequence of \cite[Theorem~1.3 and~1.6]{Ursu}. 
\end{proof}
\begin{remark}
A subgroup $H$ of $G$ is said to be \emph{plump} in $G$ if a relative version of Powers' averaging property holds 
(see \cite{Amrutam} and \cite[Definition~1.1]{Ursu}).
In general, plumpness of $H$ in $G$ implies $C^*$-irreducibility of $H\leq G$ \cite[Theorem~5.3]{Rordam},
and is equivalent to relative $C^*$-simplicity of $H\leq G$ when $H$ is normal
\cite[Theorem~1.6]{Ursu}. It therefore follows from Theorem \ref{FCGH} that plumpness of $H$ in $G$ is equivalent to $C^*$-irreducibility of $H\leq G$ when $H$ is normal.
\end{remark}
Let $\Gamma$ be a $C^*$-simple group, and let $K= \operatorname{Aut}(\Gamma)$ denote the group of all automorphisms of $\Gamma$. It is shown in \cite[Corollary 3.7]{Bedos91} that $K$ is $C^*$-simple too. As $\Gamma$ has trivial center, we may identify $\Gamma$ with the normal subgroup of $K$ consisting of all inner automorphisms of $\Gamma$, and $C_K(\Gamma)$ is then trivial. Theorem \ref{FCGH} therefore gives:
\begin{corollary} 
Let $\Gamma$ be a $C^*$-simple group. Then $\Gamma \leq \operatorname{Aut}(\Gamma)$ is $C^*$-irreducible.
\end{corollary}
Similarly, if  $\Gamma$ is $C^*$-simple and  we let it act on itself by inner automorphisms, then we have that $\Gamma \rtimes \Gamma \simeq \Gamma \times \Gamma $ is $C^*$-simple, so we get that  $\Gamma \times \Gamma \leq \Gamma\rtimes \operatorname{Aut}(\Gamma)$ is $C^*$-irreducible (cf.~the proof of \cite[Corollary 3.8]{Bedos91}).

\begin{corollary} \label{twistedC*}
Assume $H$ is  normal in $G$ and $H \leq G$ is $C^*$-irreducible. Then  $(H\leq G, \sigma)$ is $C^*$-irreducible.
\end{corollary}
\begin{proof} By Theorem \ref{FCGH}, $H$ is $C^*$-simple and $C_G(H)$ is trivial. Using Corollary \ref{C*s-sbgp} with $A= \mathbb{C}$, we get that $(H\leq G, \sigma)$ is $C^*$-irreducible.
\end{proof}

\begin{remark}\label{counterexample}
When $H$ is not normal in $G$, the equivalence between (i) and (ii) in Theorem \ref{relative-simplicity} does not hold in general, even if $\sigma$ is trivial. Here is an example where $H$ and $G$ are both $C^*$-simple, $G$ is icc relatively to $H$, but $H\leq G$ is not $C^*$-irreducible. Consider
\[
H=\left\langle
\begin{bmatrix}
	1 & 2 \\ 0 & 1
\end{bmatrix},
\begin{bmatrix}
	1 & 0 \\ 2 & 1
\end{bmatrix}
\right\rangle
\leq
\operatorname{SL}(2,\Z).
\]
It is well-known that $H\simeq\F_2$ ($H$ is sometimes called the Sanov subgroup of $\operatorname{SL}(2,\Z)$).

Furthermore, set
\[
K=\Z^2\rtimes H \quad \text{ and } \quad G=K*\Z,
\]
where $H$ acts on $\Z^2$ in the natural way. Then $H\leq K\leq G$, both $H$ and $G$ are $C^*$-simple, but $K$ is not $C^*$-simple since it contains a nontrivial normal abelian subgroup.

We now check that $G$ is icc relatively to $H$.
Let $x=(x_1,x_2)\in\Z^2$ and let $u$ and $v$ be the generators of $H$ inside $K$.
Then
\[
u^nxu^{-n}=(x_1+2nx_2,x_2) \quad \text{ and } \quad v^nxv^{-n}=(x_1,x_2+2nx_1),
\]
and it follows that $x^H=\{yxy^{-1} : y\in H\}$ is infinite for all $x\in\Z^2\setminus\{0\}$.
Next, consider $xy\in K$ where $y\in H\setminus\{e\}$.
We can always pick $w$ to be one of $u,u^{-1},v,v^{-1}$ such that $w^nyw^{-n}$ has no cancellation for $n\geq 1$.
Then
\[
\lvert (xy)^H \rvert \geq \lvert \{ w^nxyw^{-n} : n\geq 1 \} \rvert = \lvert \{ w^nxw^{-n}\cdot w^nyw^{-n} : n\geq 1 \} \rvert = \infty.
\]
One checks in a similar way that any word in $K$ and $\Z$, with at least one letter from $\Z$, has an infinite conjugacy class w.r.t.~$H$.
Thus, $g^H$ is infinite for all $g\in G\setminus\{e\}$.
\end{remark}

\begin{corollary} \label{prime-irred} Let $H$ be a normal subgroup of $G$ and assume that $H$ is prime. Then $(H\leq G,\sigma)$ is $C^*$-irreducible if and only if  $(H,\sigma)$ is $C^*$-simple and $C_G^\sigma(H)$ is trivial.
\end{corollary}
\begin{proof}
The assertion follows by combining Theorem \ref{relative-simplicity} with Theorem \ref{sigma-centralizer2}. 
\end{proof}

\begin{corollary}\label{FC-C*-simple-sigma} 
 Let $H$ be a normal subgroup of $G$. 
 
If $H$ is prime and $FC$-hypercentral,
then $(H\leq G,\sigma)$ is $C^*$-irreducible if and only if $(H,\sigma)$ satisfies Kleppner's condition  and $C_G^\sigma(H)$ is trivial.

If $H$ is $C^*$-simple, then $(H\leq G,\sigma)$ is $C^*$-irreducible if and only if $C_G^\sigma(H)$ is trivial.

\end{corollary}
\begin{proof}
We note that if $H$ is $C^*$-simple, then $H$ is icc, so $H$ is  prime and $(H,\sigma)$ satisfies Kleppner's condition. Both assertions therefore follows by combining Corollary \ref{FC-normal} with Theorem \ref{sigma-centralizer2}.
\end{proof}
\section{Some examples}  \label{irred-examples}

\begin{example}[Noncommutative tori I] \label{nct-I}
Let $\theta \in \R$ and let $A_\theta$ be the associated noncommutative $2$-torus, i.e., the universal unital $C^*$-algebra generated by two unitaries $U_1, U_2$ satisfying the relation
	\[U_1 U_2 = e^{i2\pi \theta}U_2U_1 \]
	As is well-known, we may assume that $A_\theta= C_r^*(\Z^2, \sigma_\theta)$, where 
	$\sigma_\theta$ is the two-cocycle on $\Z^2$ given by 
	\[\sigma_\theta(\mathbf{x}, \mathbf{y}) = e^{i\pi\,\theta (x_1y_2-x_2y_1)}\]
	for all $\mathbf{x}=(x_1, x_2), \mathbf{y}=(y_1, y_2) \in \Z^2$. Moreover, $A_\theta$ is simple if and only if $(\Z^2, \sigma_\theta)$ satisfies Kleppner's condition, which happens if and only if $\theta$ is irrational.
	
	Now, let $p, q \in \N$ and set $H_{p,q} := p\Z\times q\Z \leq \Z^2$, $B_{p, q, \theta} := C_r^*(H_{p, q}, \sigma_\theta)$. One readily checks that $B_{p, q, \theta} \simeq A_{pq\theta}$. Further, we have that
	the inclusion $B_{p, q, \theta} \subseteq A_\theta$ is $C^*$-irreducible if (and only if) $\theta$ is irrational. 
	
	Indeed, let $\theta$ be irrational.  Since $H_{p,q}$ is abelian, it suffices to show that $(H_{p,q} \leq \Z^2, \sigma_\theta)$ satisfies the relative Kleppner condition, cf.~Corollary \ref{FC-normal}. Let $\mathbf{x} = (x_1, x_2) \in \Z^2$. Then a short computation gives that  $\mathbf{x}$ is $\sigma_\theta$-regular w.r.t.~$H_{p,q}$ if and only if \[ \theta \,(qx_1m_2-pm_1x_2) \in \Z\] for all 
	$  \mathbf{y}= (pm_1, qm_2) \in H_{p,q}$, which is clearly possible if and only if $\mathbf{x}=(0,0)$. Thus $(0,0)$ is the only element of $\Z^2$ which is $\sigma_\theta$-regular w.r.t.~$H_{p,q}$, and the desired conclusion follows. 
	Using the final assertion of Theorem \ref{relative-simplicity}, we get that there is a one-to-one correspondence between intermediate $C^*$-algebras of $B_{p, q, \theta} \subseteq A_\theta$ and subgroups of $\Z_p \times \Z_q$. In particular, there are no strict intermediate $C^*$-algebras in this inclusion if for example $p=1$ and $q$ is prime.
\end{example}
\begin{example}[Noncommutative tori II]
	Let $\mathbf{\Theta}= (\theta_1, \theta_2, \theta_3) \in \R^3$ and let $A=A_{\mathbf{\Theta}}$ be the associated noncommutative $3$-torus, i.e., the universal unital $C^*$-algebra generated by three unitaries $U_1, U_2, U_3$ satisfying the relations
	\[U_1 U_2 = e^{i2\pi \theta_3} U_2 U_1, \,\, U_2 U_3 = e^{i2\pi \theta_1} U_3 U_2, \,\, U_3 U_1 = e^{i2\pi \theta_2} U_1 U_3\,.\]
	Then $A$ is simple if and only if the dimension 
	$d(\mathbf{\Theta})$ of $\mathbb{Q} + \theta_1\mathbb{Q}+ \theta_2\mathbb{Q} + \theta_3\mathbb{Q}$ (as a vector space over $\mathbb{Q}$) is $3$ or $4$ (see e.g.~\cite{BedML}). 
	
	Now, let  $B$ be the $C^*$-subalgebra of $A$ generated by $U_1$ and $U_2$. We may then deduce from Theorem \ref{relative-simplicity} that the following statements are equivalent:
	\begin{itemize}
		\item[i)] $B \subseteq A$ is $C^*$-irreducible,
	
	\smallskip 	\item[ii)] $\theta_3 \not \in \mathbb{Q}$ and  $d(\mathbf{\Theta}) \in \{ 3, 4\}$.
	\end{itemize}
	
\smallskip \noindent 	Indeed, we may assume that $A= C_r^*(\Z^3, \sigma_ \mathbf{\Theta})$, where 
	$\sigma_ \mathbf{\Theta}$ is the two-cocycle on $\Z^3$ given by 
	\[\sigma_ \mathbf{\Theta}(\mathbf{x}, \mathbf{y}) = e^{i\pi\, \mathbf{\Theta}\cdot(\mathbf{x} \times \mathbf{y})}\]
	for all $\mathbf{x}, \mathbf{y} \in \Z^3$,
	$ \mathbf{\Theta}\cdot(\mathbf{x} \times \mathbf{y})$
	denoting the scalar triple product.
	
	Letting $ H$ be the subgroup of $\Z^3$ given by $H= \Z^2\times \{0\}= \{ (y_1, y_2, 0) : y_1, y_2 \in \Z\}$, 
	we have that $B = C_r^*(H, \sigma_\mathbf{\Theta})$ is isomorphic to the noncommutative $2$-torus $A_{\theta_3}$ associated to $\theta_3$. Thus $B$ is simple if and only if it $\theta_3 \not\in\mathbb{Q}$.   
	
	Since $\Z^3$ is abelian, $(H\leq \Z^3,  \sigma_\mathbf{\Theta})$ satisfies the relative Kleppner's condition if and only if 
	$(0,0,0)$ is the only element of $\Z^3$ which is $\sigma_ \mathbf{\Theta}$-regular w.r.t.~$H$. 
	Now, $\mathbf{x}=(x_1, x_2, x_3) \in \Z^3$ is $\sigma_ \mathbf{\Theta}$-regular w.r.t.~$H$ if and only if 
	\[ \sigma_ \mathbf{\Theta}(\mathbf{x}, \mathbf{y})= \sigma_ \mathbf{\Theta}(\mathbf{y}, \mathbf{x})\]
	for all $\mathbf{y} \in H$. This is equivalent to
	$e^{i 2\pi \,\mathbf{\Theta}\cdot (\mathbf{x}\times \mathbf{y})} = 1$
	for all $\mathbf{y} \in H$, i.e., 
	$\mathbf{\Theta}\cdot (\mathbf{x} \times \mathbf{y}) \in \Z$ for all $\mathbf{y} \in H$. 
	As $\mathbf{\Theta}\cdot (\mathbf{x} \times \mathbf{y}) = \mathbf{y}\cdot (\mathbf{\Theta}\times \mathbf{x})$, it follows readily that 
	$(H\leq\Z^3, \sigma_\mathbf{\Theta})$ satisfies the relative Kleppner's condition 
	if and only if  the only element $\mathbf{x}$ of $\Z^3$ satisfying
	\begin{equation} \label{integers} \theta_2 x_3 - \theta_3 x_2 \in \Z \, \text{ and }\, \theta_1 x_3 - \theta_3 x_1 \in \Z\,
	\end{equation}
	is $\mathbf{x}=(0,0,0)$. 
	
	Assume now that $(H\leq \Z^3, \sigma_\mathbf{\Theta})$ does not satisfy the relative Kleppner's condition. Using what we just have shown, we can find  $\mathbf{x}=(x_1, x_2, x_3) \neq (0,0,0)$ such that \eqref{integers} holds. If $x_3 \neq 0$, then we get that $\theta_1, \theta_2 \in \operatorname{Span}_{\mathbb{Q}}\{1, \theta_3\}$, hence that $d(\mathbf{\Theta}) \leq 2$. On the other hand, if $x_3=0$, then $x_1\neq 0$ or $x_2\neq 0$, and we get that $\theta_3 \in \mathbb{Q}$. This shows that Condition ii) does not hold. 
	
	Combining this with Theorem \ref{relative-simplicity}, we get the following chain of implications:
	\begin{align*} \quad \quad \text{Condition ii) holds } &\Rightarrow \, \, (H\leq \Z^3, \sigma_\mathbf{\Theta}) \, \text{ satisfies the relative Kleppner's condition}\\
		&\Rightarrow \, \, B \subseteq A \text{ is $C^*$-irreducible}\\
		&\Rightarrow \, \, B \text{ and } A \, \text{ are simple }\\
		&\Rightarrow \, \,  \text{Condition ii) holds},
	\end{align*}
	which proves the asserted equivalence. 	
	
	When Condition ii) holds, then the last assertion in Theorem \ref{relative-simplicity} gives that the intermediate $C^*$-algebras of the inclusion $B\subseteq A$ are the noncommutative tori of the form $C^*(U_1, U_2, U_3^{\, n})$ for 
	$n \in \N \cup \{0\}$.
	
	In the case where $A$ is an higher-dimensional noncommutative $n$-torus ($n\geq 4$), and  $B$ is the $C^*$-subalgebra of $A$ generated by some of the unitary generators of $A$,  it is not easy to describe explicitly when the inclusion $B\subseteq A$ will be $C^*$-irreducible. We note that in such a situation, the necessary assumption that both $A$ and $B$ are simple is not sufficient 
 in general. Indeed, if $\theta$ and $\theta'$ are irrational numbers, and we set $B = A_\theta$, $A= A_\theta\otimes A_{\theta'}$, then $A$ and $B$ are both simple, but $B\subseteq A$ is not $C^*$-irreducible.
	
\end{example}

\begin{example}[The Heisenberg group]\label{Heisenberg}
Let $G$ denote the discrete Heisenberg group, i.e., $G = \Z^3$ as a set, with multiplication given by \[(a_1,a_2,a_3)(b_1,b_2,b_3)=(a_1 +b_1,a_2 +b_2,a_3 +b_3 +a_1b_2).\]
Let $\sigma$ be a two-cocycle on $G$. As shown in \cite[Proposition 1.1]{Pac}, $\sigma$ is (up to similarity) determined by two parameters 
$\gamma, \theta \in [0,1)$ such that
\[ \sigma\big((a_1,a_2,a_3),(b_1,b_2,b_3)\big)=
e^{i2\pi \gamma(b_3a_1+\frac{1}{2}b_2a_1(a_1-1))}e^{i2\pi \theta(a_2(b_3+a_1b_2)+\frac{1}{2}a_1b_2(b_2-1))}
\] 
The restriction of $\sigma$ to $H := \{ (0, a_2, a_3) : a_2, a_3 \in \Z\} \simeq \Z^2$ depends only on $\theta$.
It is straightforward to check that the $H$-conjugacy class $(a_1,a_2,a_3)^H$ is finite if and only if $a_1= 0$, i.e., $(a_1,a_2,a_3)\in H$.
It follows readily that $(H\leq G,\sigma)$ satisfies the relative Kleppner condition  if and only if $\theta$ is irrational. Since $H$ is abelian,
 Corollary \ref{FC-normal} gives that $(H\leq G,\sigma)$ is $C^*$-irreducible
  if and only if $\theta$ is irrational. In this case, the intermediate algebras of $C_r^*(H, \sigma)\subseteq C_r^*(G, \sigma)$ are then classified by  $\N \cup \{0\}$ via the map $n\mapsto C_r(\Gamma_n, \sigma)$,  where $\Gamma_n:= \{ (na_1,a_2, a_3) : a_1, a_2, a_3 \in \Z\} \leq G$.

Alternatively, we may consider $S^\sigma(H) := \{h \in H : h \text{ is $\sigma$-regular w.r.t.~$H$}\}$. Since $H$ is abelian, we have that $(H, \sigma)$ satisfies Kleppner's condition if and only if $S^\sigma(H)$ is trivial. On the other hand, one readily checks that $C_G(H) = H$, so we get that $C_G^\sigma(H) = S^\sigma(H)$. As $H$ is prime and FC-hypercentral, Corollary \ref{FC-C*-simple-sigma} gives that $(H\leq G,\sigma)$ is $C^*$-irreducible if and only if $S^\sigma(H)$ is trivial. Now, a short computation gives that 
\[
S^\sigma(H) = \{ (0, a_2, a_3) \in H : e^{i2\pi \theta(a_2b_3-a_3b_2)} =1 \text{ for all $b_2, b_3\in \Z$}\},
\]
which is trivial if and only if $\theta$ is irrational, in accordance with what we found above.

Note that $C_r^*(H, \sigma)= C^*(\lambda_\sigma(0,1,0), \lambda_\sigma(0,0,1))\simeq A_\theta$. If $\gamma=0$, then $C_r^*(G, \sigma) \simeq A_\theta \rtimes_r \Z \simeq A_\theta \rtimes \Z$ for the action of $\Z$ on $A_\theta$ implemented by the $*$-automorphism 
  of $A_\theta$ given by conjugation with $\lambda_\sigma(1,0,0)$. 
In the natural action of $\operatorname{SL}(2, \Z)$ on $A_\theta$, this automorphism corresponds to the matrix $\left[\begin{array}{cc} 1 & 0\\  1& 1\end{array}\right]$. 
 Actions of $\Z$ on $A_\theta$ associated to other infinite cyclic subgroups of $\operatorname{SL}(2, \Z)$ are studied  in \cite[Section~5]{ER}.
\end{example}

\begin{example}[The braid group on infinitely many strands]
The braid group $B_\infty$ on infinitely many strands is the group generated by $\{s_i\}_{i=1}^\infty$ subject to relations
\[
\begin{gathered}
	s_is_{i+1}s_i=s_{i+1}s_is_{i+1} \quad \text{ for all } i\geq 1, \\
	s_is_j=s_js_i \quad \text{ when } \lvert i-j \rvert \geq 2.
\end{gathered}
\]
There is a surjection from $B_\infty$ onto the infinite symmetric group $S_\infty$, sending the generator $s_i$ to the permutation $(i,i+1)$, and thus $s_i^2$ to the identity for every $i$.
The pure braid group $P_\infty$ is defined as the kernel of the map $B_\infty \to S_\infty$.

It follows from \cite[Proof of Theorem~6.2]{braids} that $P_\infty$ is $C^*$-simple and $B_\infty$ is icc relatively to $P_\infty$.
Hence, $P_\infty\leq B_\infty$ is $C^*$-irreducible by Theorem~\ref{FCGH},
and the intermediate algebras $C^*_r(P_\infty)\subseteq B\subseteq C^*_r(B_\infty)$ are classified by the subgroups of $S_\infty$.
\end{example}

\begin{example}[Wreath products] Let $N$ and $K$ be nontrivial groups. Recall that the restricted wreath product  $N \wr K$ is defined as the semidirect product $(\bigoplus_K N) \rtimes K$, where $K$ acts by (left) translation on the index set $K$. The unrestricted wreath product is defined in a similar way by $N\, \bar{\wr} \,K= (\prod_K N) \rtimes K$.
Then the following conditions are equivalent:
\begin{itemize}
\item[(i)] $N$ is $C^*$-simple.
\item[(ii)] $N\wr K$ is $C^*$-simple.
\item[(iii)] $\bigoplus_K N \, \leq\,  N\wr K$ is $C^*$-irreducible.
\item[(iv)]  $\bigoplus_K N\, \leq \, N\, \bar{\wr} \,K$ is $C^*$-irreducible.
\end{itemize}
Indeed, the implications  (iv) $\Rightarrow $ (iii) $\Rightarrow $ (ii) are trivial and since $N$ is subnormal  in $N\wr K$ we get (ii)~$\Rightarrow$~(i).
Now, assume that (i) holds. 
Then \cite[Theorem 5.6]{Ursu}  implies that $\bigoplus_K N$ is relatively $C^*$-simple in $N\, \bar{\wr}\, K$, so Theorem \ref{FCGH} gives that (iv) holds.

This example is also discussed in \cite[Proposition~5.5]{SUT}, where (i)~$\Rightarrow$~(ii) is shown.
\end{example}

\begin{example} \label{F2Z2}
Let $G= \mathbb{F}_2 \times \Z_2$ with $\Z_2=\{ 0, 1\}$, and let $a, b$ denote the two generators of the free group $\mathbb{F}_2$. Any element of $\mathbb{F}_2$ may be written in an unique way in reduced form as
\[ x= a^{m_1} b^{n_1} \cdots a^{m_k} b^{n_k} \]
for some $k \geq 1$, where $m_1, n_k \in \Z$, and $m_2, \ldots , m_k, n_1, \ldots, n_{k-1} \in \Z\setminus\{0\}$ if $k\geq 2$. Define then 
\[\theta_1(x):= \sum_{j=1}^k m_j, \quad \theta_2(x):= \sum_{j=1}^k n_j \quad \text{and} \quad  \theta_3(x):= \sum_{j=1}^k (m_j + n_j).\]  
For  $j=1, 2, 3$ we may then define a two-cocycle $\sigma_j$ on $G$ by  
\[ \sigma_j\big((x, k), (y, l)\big)  = \begin{cases} -1 \quad \text{if } k =1  \text{ and } \, \theta_j(y)  \, \text{ is odd,}\\
\, \, \,1 \quad \quad \text{ otherwise} \end{cases}\]
for $(x, k), (y, k) \in \mathbb{F}_2 \times \Z_2$. (One may show that every two-cocycle on $G$ which is not similar to $1$ is similar to $\sigma_j$ for some $j\in \{1, 2, 3\}$, but we do not need this.)

Set $H= \mathbb{F}_2 \times \{0\}$. Then $H$ is  $C^*$-simple, and $C_G(H) = \{e\} \times \Z_2$. For  $y\in \mathbb{F}_2$ and $j \in \{1, 2, 3\}$, we have
\[  \sigma_j\big((e, 1), (y, 0)\big) = \begin{cases} -1 \quad \text{if } \theta_j(y)  \, \text{ is odd,}\\
\, \, \,1 \quad \quad \text{ otherwise} \end{cases}, \, \text{ while } \, \sigma_j\big((y, 0), (e, 1)\big)=1. \]
Thus $(e,1)$ is not $\sigma_j$-regular w.r.t.~$H$. It follows that $C_G^{\sigma_j}(H) = \{ (e, 0)\}$  for $j= 1, 2, 3$.
Hence, Corollary \ref{FC-C*-simple-sigma} gives that $(H\leq  G, \sigma_j)$ is $C^*$-irreducible for $j=1, 2, 3$. 
\end{example}

\section{$C^*$-irreducibility and groups acting on trees}  \label{irred-trees}

Let $T$ be a tree and $\partial T$ its boundary \cite{serre}. 
An automorphism $g\in \operatorname{Aut}(T)$ is \emph{elliptic} if it fixes a vertex of $T$,
is an \emph{inversion} if it does not fix any vertices, but permutes two adjacent vertices, i.e., inverts an edge,
and is called \emph{hyperbolic} if it is not elliptic nor an inversion.
The fixed point set $T^g$ of an elliptic automorphism $g$ of $T$ is a (possibly finite) subtree of $T$.
An hyperbolic automorphism $h$ does not fix any vertices,
but has an axis $L^h$, which is an infinite linear subtree on which $h$ acts by translation. 
Moreover, $h$ admits exactly two fixed points in $T\cup\partial T$, namely the two points in $\partial T$ arising from the $h$-invariant axis $L^h$.
Two hyperbolic automorphisms are said to be \emph{transverse} if they have disjoint fixed point sets.
We refer to \cite[Proposition~6.4.24]{serre} for details.

An action of a discrete group $G$ on a tree $T$ is \emph{minimal} if $T$ does not contain any proper $G$-invariant subtrees,
and of \emph{general type} (or \emph{strongly hyperbolic} \cite{HP}) if its image in  $\operatorname{Aut}(T)$ contains two transverse hyperbolic automorphisms of $T$. 
The following result can be found in \cite[Proposition 3]{LB};
see \cite[Lemma~2.10]{GGT} for a related result.

\begin{lemma}\label{normal-on-tree}
Let $G$ act faithfully on a tree $T$ and assume that the action is minimal and of general type.
Let $H$ be a nontrivial normal subgroup of $G$.
Then the action of $H$ on $T$ is also minimal and of general type.
\end{lemma}

The following consequence of Lemma~\ref{normal-on-tree} is surely part of the folklore, but we include a proof for completeness.
\begin{proposition}\label{triv-cent}
Let $G$ act faithfully on a tree $T$ and assume that the action is minimal and of general type.
Let $H$ be a nontrivial normal subgroup of $G$.
Then the centralizer $C_G(H)$ of $H$ in $G$ is trivial.
\end{proposition}

\begin{proof}
Since the action of $H$ on $T$ is of general type by Lemma~\ref{normal-on-tree}, $H$ contains a hyperbolic element $h$. Let $g\in C_G(H)$. Then $gh=hg$, and the axis $L^h$ is invariant under the action of $g$.
Indeed, we have $g L^h = g h L^h = h g L^h$, so $g L^h$ is another infinite $h$-invariant axis, and so $g L^h = L^h$. 
Thus $L^h$ is invariant under the action of $C_G(H)$ on $T$. The action of $G$ on $T$ being of general type, we have $L^h\neq T$. Thus the action of $C_G(H)$ on $T$ is not minimal. Since $H$ is normal in $G$, $C_G(H)$ is normal in $G$ too, so Lemma~\ref{normal-on-tree} gives that $C_G(H)$ must be trivial.
\end{proof}

If $G$ is the free product of two groups not both of order $2$ and $H$ is a normal subgroup of $G$, then \cite[Theorem~5.2]{Ursu} says that $H$ is relatively $C^*$-simple in $G$, hence that $H\leq G$ is $C^*$-irreducible by Theorem \ref{FCGH}. This may be generalized as follows. 
\begin{theorem} \label{tree-C*s}
Assume $G$ has a faithful, minimal action of general type on a tree $T$, and let $H$ be a nontrivial normal subgroup of $G$.
Then the following conditions are equivalent:
\begin{itemize}
\item[(i)] $G$ is $C^*$-simple.
\item[(ii)] $H$ is $C^*$-simple.
\item[(iii)] $H\leq G $ is $C^*$-irreducible.
\item[(iv)] $(H\leq G, \sigma) $ is $C^*$-irreducible for every two-cocycle $\sigma$ on $G$.
\end{itemize}
\end{theorem}
\begin{proof} Since $C_G(H)$ is trivial by Proposition \ref{triv-cent}, (i) is equivalent to (ii) by \cite[Theorem 1.4]{BKKO}, while Theorem \ref{FCGH} gives that (ii) and (iii) are equivalent. The equivalence between (iii) and (iv) follows from Corollary \ref{twistedC*}. 
\end{proof}

We recall that an action of a group $G$ on a set $X$ is called \emph{strongly faithful} if for any finite subset $F \subseteq G \setminus \{e\}$, there exists $x \in X$ such that $fx\not= x$ for all $f \in F$. Also, an action of $G$ on a topological space $Y$ is said to be \emph{topologically free} if the set $Y^g:=\{y \in Y: gy = y\}$ has empty interior for each $g\in G \setminus \{e\}$.  Now, consider a minimal action of $G$ of general type on a tree $T$. As shown in  \cite[Proposition~3.8]{BIO}, such an action is strongly faithful if and only if the induced action of $G$ on $\partial T$ (equipped with the relative shadow topology) is topologically free, if and only if the induced action of $G$ on the closure $\overline{\partial T}$ (of $\partial T$ in $T\cup \partial T$ w.r.t.~the shadow topology) is topologically free. Moreover,  $G$ is $C^*$-simple (in fact a Powers group) when one of these conditions holds.  
Hence, Theorem \ref{tree-C*s} gives:
\begin{corollary} \label{str-f-C^*}
Assume that $G$ has a strongly faithful, minimal action of general type on a tree $T$,  
and let $H$ be a nontrivial normal subgroup of $G$.Then $H\leq G $ is $C^*$-irreducible. Moreover, 
$(H\leq G, \sigma) $ is $C^*$-irreducible for every two-cocycle $\sigma$ on $G$.
\end{corollary}

\begin{example} Let $G=A*_C B$ be an amalgam of groups which is  nondegenerate, i.e., at least one of the indices $[A:C], [B:C]$ is strictly larger than $2$. Then the action of $G$ on its Bass-Serre tree $T$ is minimal and of general type, cf.~\cite[Proposition 19]{HP}. Moreover, \cite[Proposition 5.3]{amalgams} gives that this action is strongly faithful (equivalently, the action of $G$ on $\partial T$ is topologically free) whenever the so-called interior group of $G$, $\operatorname{int}(G) $, is trivial. We also note that \cite[Proposition 3.2]{amalgams} characterizes in several ways when this happens, e.g.,~it is equivalent to require that for every finite subset $F$ of $C\setminus \{e\}$, there exists $g\in G$ such that $gFg^{-1} \cap C = \emptyset$.
When $\operatorname{int}(G)$ is trivial and $H$ is a nontrivial normal subgroup of $G$, we can apply Corollary \ref{str-f-C^*} and deduce that $(H\leq G, \sigma) $ is $C^*$-irreducible for any two-cocycle $\sigma$ on $G$. 

We note that if $C$ is trivial, i.e.,  $G= A*B$ with $\max\{ \lvert A\rvert, \lvert B\rvert \} \geq 3$, then every two-cocycle on $G$ is similar to a two-cocycle of the form $\sigma_1*\sigma_2$, where  $\sigma_1$ (resp.~$\sigma_2$) is a two-cocycle on $A$ (resp.~$B$), see for instance \cite[Section~4]{primeness} and references therein. It is therefore easy  to obtain concrete examples $(H \leq A*B, \sigma)$ that are $C^*$-irreducible. A natural choice here is to let $H$ be the kernel of the canonical homomorphism from $A*B$ into $A\times B$.
\end{example}
 
\begin{example}
Similarly, let $G= \operatorname{HNN}(\Gamma, A, \theta)$ be an HNN-extension which is nonascending, i.e., we have $A\neq G \neq  \theta(A) $. Then the action of $G$ on  its Bass-Serre tree $T$ is minimal and of general type, cf.~\cite[Proposition 20]{HP} (see also \cite[Proposition 4.16]{BIO}). Moreover, \cite[Proposition 4.18]{BIO} gives that this action is strongly faithful (equivalently, the action of $G$ on $\partial T$ is topologically free) whenever the interior group of $G$, $\operatorname{int}(G) $, is trivial. Several characterizations of $\operatorname{int}(G) $ being trivial are given in \cite[Theorem 4.10]{BIO}, e.g., this happens if and only if  for every finite subset $F$ of $A\setminus \{e\}$, there exists $g\in G$ such that $gFg^{-1} \cap A = \emptyset$. When $\operatorname{int}(G)$ is trivial, we can apply Corollary \ref{str-f-C^*} and get that $(H\leq G, \sigma)$ is $C^*$-irreducible whenever $H$ is a  nontrivial normal subgroup of $G$ and $\sigma$ is a two-cocycle on $G$. 

To illustrate this, let $m, n\in \Z$ be such that $\min\{\lvert m\rvert, \lvert n\rvert\} \geq 2$
and $\lvert m\rvert\neq \lvert n\rvert$, and let $G=BS(m,n) = \langle t, b \mid t^{-1} b^m t = b^n\rangle $ denote the associated Baumslag-Solitar group, which is an HNN-extension with $\Gamma = \Z$, $A= m\Z$, and $\theta(mk) = nk$ for $k\in \Z$. Then it is well-known that $G$ is $C^*$-simple (see for example \cite[Theorem  3]{HP} and \cite[Example 4.21]{BIO}). Moreover, $G$, as an HNN-extension, is clearly nonascending, and its interior group is trivial. Indeed, the final part of  \cite[Example 4.21]{BIO} shows that  the so-called quasi-kernel $K_1$ of $G$  is trivial, and this is equivalent to $\operatorname{int}(G)$ being trivial by \cite[Theorem 4.10]{BIO}.
Let now $\varphi\colon G \to \Z$ be the homomorphism determined by $\varphi(t) = 1$ and $\varphi(b) =0$ and set $H= SBS(m,n) = \ker \varphi$ (as in \cite[Corollary 5]{HP}). Then we can conclude from the previous paragraph that $H\leq G$ is $C^*$-irreducible.
\end{example}
If $G$ acts on compact Hausdorff space $X$ in such a way that every orbit is dense in $X$ (i.e., the action is minimal) and  the weak*-closure of every orbit in the space of probability measures on $X$ contains a point mass (i.e., the action is strongly proximal), then the action of $G$ is called a \emph{boundary action} and $X$ is called a \emph{$G$-boundary}. Up to $G$-equivariant homeomorphism, the \emph{Furstenberg boundary} $\partial_F G$ is the unique $G$-boundary having the universal property that for any $G$-boundary $X$, there exists a (unique) $G$-equivariant, continuous surjection $\partial_F G \to X$.  As shown in \cite[Theorem  6.2]{KK} (see also \cite[Theorem 3.1]{BKKO}), $G$ is $C^*$-simple if and only if there exists a topologically free $G$-boundary, 
if and only if the action of $G$ on $\partial_F G$ is free (resp.~topologically free). 

When a faithful action of $G$ on a tree $T$ is minimal and of general type,
then  $\overline{\partial T}$ is a $G$-boundary (cf.~\cite[Lemma 3.5]{BIO}; see also \cite[Proposition~4.26]{LBMB}),
and thus also an $H$-boundary for any nontrivial normal subgroup $H$ of $G$ by Lemma~\ref{normal-on-tree}.
 Corollary \ref{str-f-C^*} can therefore also be deduced from the following result for boundary actions, which relies heavily on \cite[Theorem~1.3]{Ursu} and its proof. 
\begin{proposition}
Let $H$ be a nontrivial normal subgroup of $G$ and suppose that there 
exists a topologically free boundary action of $G$ on $X$ which restricts to a boundary action of $H$.  
Then $(H\leq G, \sigma) $ is $C^*$-irreducible for every two-cocycle $\sigma$ on $G$.
\end{proposition}

\begin{proof}
Since $X$ is an $H$-boundary, there exists an $H$-equivariant, continuous surjection map $\partial_F H \to X$. 
By~\cite[Lemma 5.2]{BKKO}, the action of $H$ on $\partial_F H$ extends in a unique way to an action of $G$ on $\partial_F H$, so $\partial_F H$ is a $G$-boundary.
Further, it follows from \cite[Corollary~4.3]{Ursu} that the surjection  $\partial_F H \to X$ is $G$-equivariant.
Hence, if $g\in G$ is such that $(\partial_F H)^g$ has nonempty interior, then $X^g$ has nonempty interior. 
By assumption, this means that $g=e$. Thus, the action of $G$ on $\partial_F H$ is topologically free, and 
\cite[Theorem~1.3]{Ursu} gives that $H$ is $C^*$-simple and $C_G(H)$ is trivial.
The conclusion follows  from Theorem~\ref{FCGH} and Corollary \ref{twistedC*}. 
\end{proof}

\section{On $C^*$-simplicity and normal subgroups}

Let $H$ be a normal subgroup of $G$. As recalled in Section \ref{irred-rtgc}, we then have that $G$ is $C^*$-simple if and only if $H$ and $C_G(H)$ are both $C^*$-simple, cf.~\cite[Theorem 1.4]{BKKO}. It would be interesting to know what kind of assumptions are needed to ensure that a twisted version of this result holds. Our goal in this section is to prove a result in this direction, cf.~Corollary \ref{FC-sigma-inv}.  

\begin{lemma} \label{prime-Kleppner}
Let $H$ be a normal subgroup of $G$. Assume that $H$ is prime and that 
$(H,\sigma)$ and $(C_G^\sigma(H),\sigma)$ both satisfy Kleppner's condition.
Then $(G,\sigma)$ satisfies Kleppner's condition.
\end{lemma}

\begin{proof}
Suppose that $g\in G$ is $\sigma$-regular w.r.t.~$G$ and $\lvert g^G\rvert < \infty$.
Then $g\in FC_G^\sigma(H)$, so Lemma~\ref{sigma-centralizer} gives that $g\in C_G^\sigma(H)$. Now, it is clear that 
$g$ is also $\sigma$-regular w.r.t.~$C_G^\sigma(H)$ and $\lvert g^{C^\sigma_G(H)}\rvert < \infty$. 
Since $(C_G^\sigma(H),\sigma)$ satisfies Kleppner's condition,
we get that $g=e$. This shows that $(G,\sigma)$ satisfies Kleppner's condition.
\end{proof}

\begin{lemma}\label{normal-hypercentral}
Suppose that $H$ is a nontrivial normal subgroup of a $FC$-hypercentral group $G$.
Then $H\cap FC(G)$ is nontrivial.
\end{lemma}

\begin{proof}
We recall, see for example \cite[Section~4.3]{Robinson}, that $G$ is $FC$-hypercentral if and only if $G$ is equal to its $FC$-hypercenter $FCH(G)$, which is defined  as follows.  Let  $\{F_\alpha\}_\alpha$ 
be the ascending  series of normal subgroups of $G$ indexed by the ordinal numbers, given by
$F_0=\{e\}$, $F_\alpha/F_\beta =FC(G/F_\beta)$ if $\alpha=\beta+1$,
and $F_\alpha=\bigcup_{\beta<\alpha}F_\beta$ when $\alpha$ is a limit ordinal number.
This series eventually stabilizes and $FCH(G):=\lim_\alpha F_\alpha=\bigcup_\alpha F_\alpha$. 

Now, since $G= FCH(G)$, there is some $\alpha$ such that $H\cap F_\alpha=\{e\}$ while $H\cap F_{\alpha+1} \neq \{e\}$.
Pick  $h\in H\cap F_{\alpha+1}$, $h\neq e$.
As $H\cap F_\alpha=\{e\}$, the homomorphism $h'\mapsto h'F_\alpha$ from $H$ into $G/F_\alpha$ is injective. Also,  $hF_\alpha$ belongs to $ FC(G/F_\alpha)$ since $h \in F_{\alpha+1}$. We therefore get that
\[
\lvert\{ ghg^{-1} : g\in G \}\rvert = \lvert\{ ghg^{-1}F_{\alpha} : g\in G \}\rvert =  \lvert\{ (gF_\alpha) hF_\alpha (gF_{\alpha})^{-1} : g\in G \}\rvert< \infty.
\]
Hence, $h\in FC(G)$. Thus we have $e\neq h \in H\cap FC(G)$.
\end{proof}

\begin{definition} 
A subgroup $H$ of $G$ will be said to be \emph{$\sigma$-regular in $G$} if 
$h\in H$ is $\sigma$-regular w.r.t.~$G$ whenever $h$ is $\sigma$-regular w.r.t.~$H$.
\end{definition}

\begin{example} 
Let $\theta$ be irrational,  $p,q \in \N$, $H_{p, q} = p\Z \times q\Z$ and $\sigma_\theta$ be the two-cocycle on $\Z^2$ defined in Example \ref{nct-I}. Then  $H_{p, q}$ is $\sigma$-regular in $\Z^2$. Indeed, $\mathbf{y}=(0,0)$ is the only element of  $H_{p,q}$ which is $\sigma$-regular w.r.t.~$H_{p, q}$ (because $(H_{p, q}, \sigma_\theta)$ satisfies Kleppner's condition), and $(0,0)$ is obviously $\sigma$-regular w.r.t.~$\Z^2$. The same argument shows that if $H\leq G$, $H$ is  abelian and  $(H, \sigma)$ satisfies Kleppner's condition, then $H$ is $\sigma$-regular in $G$.
\end{example}

\begin{example}
Let $G= \mathbb{F}_2\times \Z_2$, $H=\mathbb{F}_2\times\{0\} $ and $\sigma_j$ be as in Example \ref{F2Z2} for $j\in \{1,2,3\}$.
Then $H$ is not $\sigma_j$-regular in $G$.
Indeed, every $(x, 0) \in H$ is $\sigma_j$-regular w.r.t.~$H$ (because $\sigma_j$ restricts to the trivial two-cocycle on $H$). But if we pick $x\in \mathbb{F}_2$ such that $\theta_j(x)$ is odd, then 
$(x,1)$ commutes with $(x, 0)$, and 
\[ \sigma_j((x,0), (x, 1)) = 1 \neq -1 =  \sigma_j((x,1), (x, 0)),\]
i.e., $(x, 0)$ is not $\sigma_j$-regular w.r.t.~$G$.
\end{example}

It is not difficult  to see that if $(G,\sigma)$ satisfies Kleppner's condition and $H$ is a $\sigma$-regular subgroup of $G$ having finite index, then $(H,\sigma)$ satisfies Kleppner's condition. If $H$ is of infinite index, this might not be true. However, we have:
\begin{proposition}\label{sigma-invariant-subgroup}
Let $G$ be a FC-hypercentral group and $H$ be a normal subgroup of $G$.
Assume that $(G,\sigma)$ satisfies Kleppner's condition and that
$H$ is prime and $\sigma$-regular in $G$.
Then $(H,\sigma)$ satisfies Kleppner's condition.
\end{proposition}

\begin{proof}
Assume $h\in H$ is $\sigma$-regular w.r.t.~$H$ and $h^H$ is finite.
Then $h$ is $\sigma$-regular w.r.t.~$G$ (since $H$ is $\sigma$-regular in $G$).
Let $N$ be the normal subgroup of $G$ generated by $h$,
i.e., the subgroup generated by $h^G$. 

Note that $FC(H)$ is a characteristic subgroup of $H$, so it is normal in $G$ (this can also be checked directly).  
It follows that $N$ is contained in $FC(H)$, so $N$ is torsion-free (since $H$ is prime). 
Moreover, using Lemma \ref{normal-hypercentral}, we get that $N\cap FC(G)$ is nontrivial.

Let $y\in N\cap FC(G)$, $y\neq e$. 
Then $y$ can be written as a finite product of elements from $h^G$ and their inverses, 
and all these elements clearly belong to $FC_G^\sigma(G) = \{ g \in FC(G) : g \text{ is $\sigma$-regular w.r.t.}~G\}$.
Thus, by Lemma~\ref{sigma-product}, there exists some $n\in \N$ such that $y^n \in FC_G^\sigma(G)$. 
As $y \in N$, and $N$ is torsion-free, we can conclude that $FC_G^\sigma(G)$ is nontrivial,
i.e., $(G,\sigma)$ does not satisfy Kleppner's condition.
\end{proof}

\begin{corollary}\label{FC-sigma-inv}
Let $G$ be a FC-hypercentral group and $H$ be a normal subgroup.
Assume that $H$ is prime and $\sigma$-regular in $G$.
Consider the following two conditions:
\begin{itemize}
	\item[(i)] $(H,\sigma)$ and $(C_G^\sigma(H),\sigma)$ are both $C^*$-simple.
	\item[(ii)] $(G,\sigma)$ is $C^*$-simple.
\end{itemize}
Then \textup{(i)}~$\Rightarrow$~\textup{(ii)}. If we also assume that $C_G^\sigma(H)$ is prime and $\sigma$-regular in $G$, then \textup{(ii)}~$\Rightarrow$~\textup{(i)}.
\end{corollary}
\begin{proof}
Assume (i) holds. Then $(H, \sigma)$ and $(C_G^\sigma(H),\sigma)$ both satisfy Kleppner's condition, so Lemma \ref{prime-Kleppner} gives that $(G,\sigma)$ satisfies Kleppner's condition. Since $G$ is $FC$-hypercentral, we get that (ii) holds. 

Next, assume  also that $C_G^\sigma(H)$ is prime and $\sigma$-regular in $G$, and that (ii) holds. Then $(G, \sigma)$ satisfies Kleppner's condition, so Proposition \ref{sigma-invariant-subgroup} gives that $(H, \sigma)$ and $(C_G^\sigma(H),\sigma)$ both satisfy Kleppner's condition. Since $H$ and $C_G^\sigma(H)$ are $FC$-hypercentral (being  subgroups of $G$), 
we get that  $(H,\sigma)$ and $(C_G^\sigma(H),\sigma)$ both are $C^*$-simple.
\end{proof}

\begin{question}
Let $H$ be a normal subgroup of $G$ and consider the following two properties:
\begin{itemize}
\item[(i)] $(H,\sigma)$ and $(C_G^\sigma(H),\sigma)$ are both $C^*$-simple.
\item[(ii)] $(G,\sigma)$ is $C^*$-simple.
\end{itemize}
Under which assumptions (other than $\sigma=1$ and those imposed in Corollary \ref{FC-sigma-inv}) do we have that (i)~$\Rightarrow$~(ii), or that (i)~$\Leftrightarrow$~(ii) ?
\end{question}

\section*{Acknowledgements} We are very grateful to the referee for carefully reading the manuscript and for pointing out several places where it could be improved.

\bibliographystyle{plain}
\bibliography{relative-simplicity}

\end{document}